\documentclass[10pt]{article}
\usepackage[utf8]{inputenc}
\usepackage[T1]{fontenc}
\usepackage{amsmath}
\usepackage{amsthm}
\usepackage{amsfonts}
\usepackage{tikz}
\usetikzlibrary{matrix, positioning, fit}
\usepackage{arydshln}
\usepackage{float}
\usepackage{amssymb}
\usepackage{subcaption}
\usepackage[a4paper, total={5.7in, 8in}]{geometry}
\usepackage[version=4]{mhchem}
\usepackage{stmaryrd}
\usepackage{bbm}
\usepackage{graphicx}
\usepackage{color}
\usepackage{bm}
\usepackage[export]{adjustbox}
\usepackage{hyperref}
\usepackage{lineno}
\graphicspath{ {./images/} }

\newtheorem{theorem}{Theorem}[section]
\newtheorem{lemma}[theorem]{Lemma}
\newtheorem{remark}[theorem]{Remark}
\numberwithin{figure}{section}
\numberwithin{table}{section}

\newcommand{\multi}{\gamma}
\newcommand{\mumu}{\mu}
\newcommand{\muI}{\mathcal{I}}
\newcommand{\muhat}{\widehat{\mumu}}

\begin{document}

\title{Periodic localized traveling waves\\ in the two-dimensional suspension bridge equation}

\author{Lindsey van der Aalst\thanks{VU Amsterdam, Department of Mathematics, De Boelelaan 1111, 1081 HV Amsterdam, The Netherlands. \href{mailto:l.j.w.van.der.aalst@vu.nl}{l.j.w.van.der.aalst@vu.nl}; partially supported by NWO grant 613.009.132 (Corresponding author)} \and Jan Bouwe van den Berg\thanks{VU Amsterdam, Department of Mathematics, De Boelelaan 1111, 1081 HV Amsterdam, The Netherlands. \href{mailto:janbouwe@few.vu.nl}{janbouwe@few.vu.nl}; partially supported by NWO grant 613.009.132} \and Jean-Philippe Lessard\thanks{McGill University, Department of Mathematics and Statistics, Burnside Hall, 805 Sherbrooke Street West, Montreal, Québec H3A 0B9, Canada. \href{mailto:jp.lessard@mcgill.ca}{jp.lessard@mcgill.ca}} }

\maketitle
\begin{abstract}
In the dynamics generated by the suspension bridge equation, traveling waves are an essential feature. The existing literature focuses primarily on the idealized one-dimensional case, while traveling structures in two spatial dimensions have only been studied via numerical simulations. We use computer-assisted proof methods based on a Newton-Kantorovich type argument to find and prove periodic localized traveling waves in two dimensions. 
The main obstacle is the exponential nonlinearity in combination with the resulting large amplitude of the localized waves. Our analysis hinges on establishing computable bounds to control the aliasing error in the computed Fourier coefficients. This leads to existence proofs of different traveling wave solutions, accompanied by small, explicit, rigorous bounds on the deficiency of numerical approximations. This approach is directly extendable to other wave equation models and elliptic partial differential equations with analytic nonlinearities, in two as well as in higher dimensions.
\end{abstract}

\section{Introduction}
The purpose of this paper is to find and prove suitably localized traveling wave solutions to the suspension bridge equation in two spatial dimensions. Historically, one of the stimuli for studying suspension bridges was the collapse of the Tacoma Narrows Bridge in 1940 \cite{tacoma_amann}. The collapse was caused by the blowing of a strong wind, creating waves in the bridge, which eventually resulted in failure of the construction. This catastrophic event motivated the development of mathematical models of suspension bridges. Overviews of the history of various mathematical models for suspension bridges can be found in \cite{Drabek_overview} and \cite{Gazzola_overview}. As argued in \cite{Drabek_overview}, in essence there are two type of models of suspension bridges: very realistic, advanced ones (e.g.~\cite{Arioli2017,McKenna_advanced,Moore_advanced}) and more simplistic models, which are better suitable for mathematical analysis. In this paper, we study the relatively simple model (see e.g.~\cite{sb_chen,SB_nagatou})
\begin{equation}\label{eq:waveeq}
    \partial^2_t U = -\Delta^2 U-e^U+1,
\end{equation}
where $U=U(X_1,X_2,t) \in \mathbb{R}$ denotes the deflection of the bridge surface in the downward direction, depending on spatial variables $X_1,X_2\in \mathbb{R}$ and time $t\geq 0$.  The bilaplacian $\Delta^2$ in \eqref{eq:waveeq} is given by
$ \Delta^2 = \left(\partial^2_{X_1}+ \partial^2_{X_2}\right)^2$. 

In numerical simulations of this partial differential equation (PDE), traveling localized structures are frequently observed (e.g.~\cite{wave_breuer,sb_chen,SB_horak}). To study these, we introduce the parameter $c \in \mathbb{R}$, representing the wave speed in the $X_1$-direction, and we reduce \eqref{eq:waveeq} further by substituting the traveling wave ansatz \begin{align}\label{eq:travwaveansatz}
	U(X_1,X_2,t)=u(X_1-ct, X_2)=u(x_1,x_2),
\end{align} 
where the traveling wave profile $u$ satisfies 
\begin{align}\label{eq:SBeq}
    \Delta^2u+c^2 \partial^2_{x_1}u&+e^u-1=0.
\end{align}    

In one spatial dimension, the traveling wave equation is an ordinary differential equation (ODE) rather than a PDE, i.e.,
\begin{align}\label{eq:1DSB}
u'''' + c^2 u''+e^u-1=0.
\end{align}
This one-dimensional suspension bridge equation has been studied extensively (e.g.~\cite{wave_breuer,sb_chen,sb_mckenna,SB_nagatou,Smets_homoclinic}). In \cite{Smets_homoclinic}, it was proven that for \textit{almost all} $c \in (0,\sqrt{2})$ there exist at least one homoclinic solution of \eqref{eq:1DSB}. These homoclinic solutions correspond to planar traveling wave solutions of \eqref{eq:waveeq} that are localized in the $X_1$-direction. In addition, existence of at least one homoclinic solution was proven for all $c \in [0.71,1.37]$ in \cite{BergBreden2018} using computer-assisted means and for all $c \in (0,0.74)$  in \cite{SantraWei} using variational methods. Complementing these findings, $36$ homoclinic solution were proven for $c=1.3$ in \cite{wave_breuer}. Orbital (in)stability of these solutions was investigated in \cite{SB_nagatou}. 

In \cite{PeletierTroy}, existence of periodic, instead of homoclinic, solutions of \eqref{eq:1DSB} was proven. The authors developed a method for proving existence of two families of (multibump) periodic solutions for all $c \in (0, \sqrt{2})$. Both families consist of sets of distinct periodic traveling patterns.

It is observed in many of the above mentioned papers that for $c$ approaching $\sqrt{2}$, the amplitudes of the waves tend to zero, leading to oscillatory solutions. On the other hand, when $c$ approaches zero, the amplitudes of the waves go to infinity. Indeed this behavior is known to hold in any dimension~\cite{Santra_stability}. At $c=0$, solutions no longer exist \cite{PeletierTroy,Santra_stability}. 

In contrast to the many analytic results for the one-dimensional case summarized above, the two-dimensional traveling wave problem has so far only been studied numerically. The simulations presented in \cite{SB_horak} suggest that localized (in both spatial directions) traveling structures appear in the dynamics of \eqref{eq:waveeq}. 
However, to our knowledge, an existence proof is lacking.

Rather than studying the idealized setting of localized solutions of \eqref{eq:SBeq} on an unbounded two-dimensional domain, which is certainly mathematically interesting but for now out of reach, we reduce technical difficulties by considering a large domain with periodic boundary conditions. We are interested in the situation where the bridge is much longer than the characteristic length scale of the observed localized traveling patterns. Physically, the length scale along the $X_1$-direction of the bridge is of primary importance. To reduce the influence of the boundary condition in the $X_2$-direction, we focus on solutions which are localized in both spatial dimensions, as was done in \cite{SB_horak}. Of course, both an infinitely long bridge and a long ``periodic'' bridge are merely mathematical idealizations.

To be precise, we introduce constants $L_1,L_2>0$. We set out to prove traveling wave solutions on a bridge of width $2L_2$ (in the $X_2$-direction) that are $2L_1$ periodic in the $X_1$-direction. 
By using computer-assisted techniques we will be able to inspect that the localization length scale of the solution found is indeed much smaller than $L_1$ and $L_2$.

The natural ``free'' boundary conditions in the $X_2$-direction for the bilaplacian are 
$\Delta U\left(\,\cdot\,,\pm L_2\right) = 0 
=\partial_{X_2} \Delta U\left(\,\cdot\,,\pm L_2\right)$. 
In this paper we further idealize the setting by instead imposing Neumann boundary conditions 
$
\partial_{X_2}U\left(\,\cdot\,,\pm L_2\right) = 0 
=\partial_{X_2} \Delta U\left(\,\cdot\,,\pm L_2\right)
$,
which are easier to deal with in our approach. 
Finally, we take advantage of the symmetry ($x_1 \mapsto -x_1$, $x_2 \mapsto -x_2$) and consider solutions on a rectangle $[0,L_1] \times [0,L_2]$ with Neumann boundary conditions on all sides:
%
%
\begin{subequations}\label{eq:SBBC}
\begin{align}      \partial_{x_1}u\left(0,x_2\right)=0, \quad \partial_{x_1}u\left(L_1,x_2\right)=0, \quad \partial^3_{x_1}u\left(0,x_2\right)=0, \quad \partial^3_{x_1}u\left(L_1,x_2\right)=0,\\         \partial_{x_2}u\left(x_1,0\right)=0, \quad \partial_{x_2}u\left(x_1,L_2\right)=0, \quad \partial^3_{x_2}u\left(x_1,0\right)=0, \quad \partial^3_{x_2}u\left(x_1,L_2\right)=0. 
       \end{align}
 \end{subequations} 
The symmetries allow a natural extension of a solution $u$ on the domain $[0,L_1] \times [0,L_2]$ satisfying the boundary conditions~\eqref{eq:SBBC} to a solution of~\eqref{eq:SBeq} on $\mathbb{R} \times [-L_2,L_2]$ through 
\begin{align*}
	u(x_1,x_2)=u(-x_1,x_2)=u(x_1,-x_2)=u(2L_1+x_1,x_2).
\end{align*} 

This extended profile $u$ then constitutes, via~\eqref{eq:travwaveansatz}, a periodic traveling wave solution of~\eqref{eq:waveeq} 
on $\mathbb{R} \times [-L_2,L_2]$ satisfying Neumann boundary conditions in the $X_2$-direction. 
 
In this paper, we develop a computer-assisted method for proving existence of solutions to \eqref{eq:SBeq} with boundary conditions \eqref{eq:SBBC}. An example of the type of result we obtain is stated in Theorem \ref{thm:examplesol}
and illustrated in Figure~\ref{fig:examplesol}. 
The method yields an explicit approximation of a \textit{periodically localized} traveling wave, i.e., periodic in the $X_1$-direction and localized in space when viewed on a fundamental domain of periodicity,
together with an also explicit and very small error bound on the distance between the proven solution and the depicted approximation.

\begin{theorem}\label{thm:examplesol}
For $c=1.1$ and $L=(50,40)$, there exists a solution $\hat{u}$ to \eqref{eq:SBeq} with boundary conditions \eqref{eq:SBBC} such that 
$\|\hat{u} - \bar{u}\|_\infty \leq 4\cdot10^{-8}$,
where $\bar{u}$ is the numerical approximation visualized in Figure \ref{fig:examplesol}. The finitely many non-vanishing Fourier-cosine coefficients of $\bar{u}$ can be found in the code at~\cite{github}. 
\end{theorem}

\begin{figure}[t]
    \centering
    \includegraphics[width=0.8\textwidth]{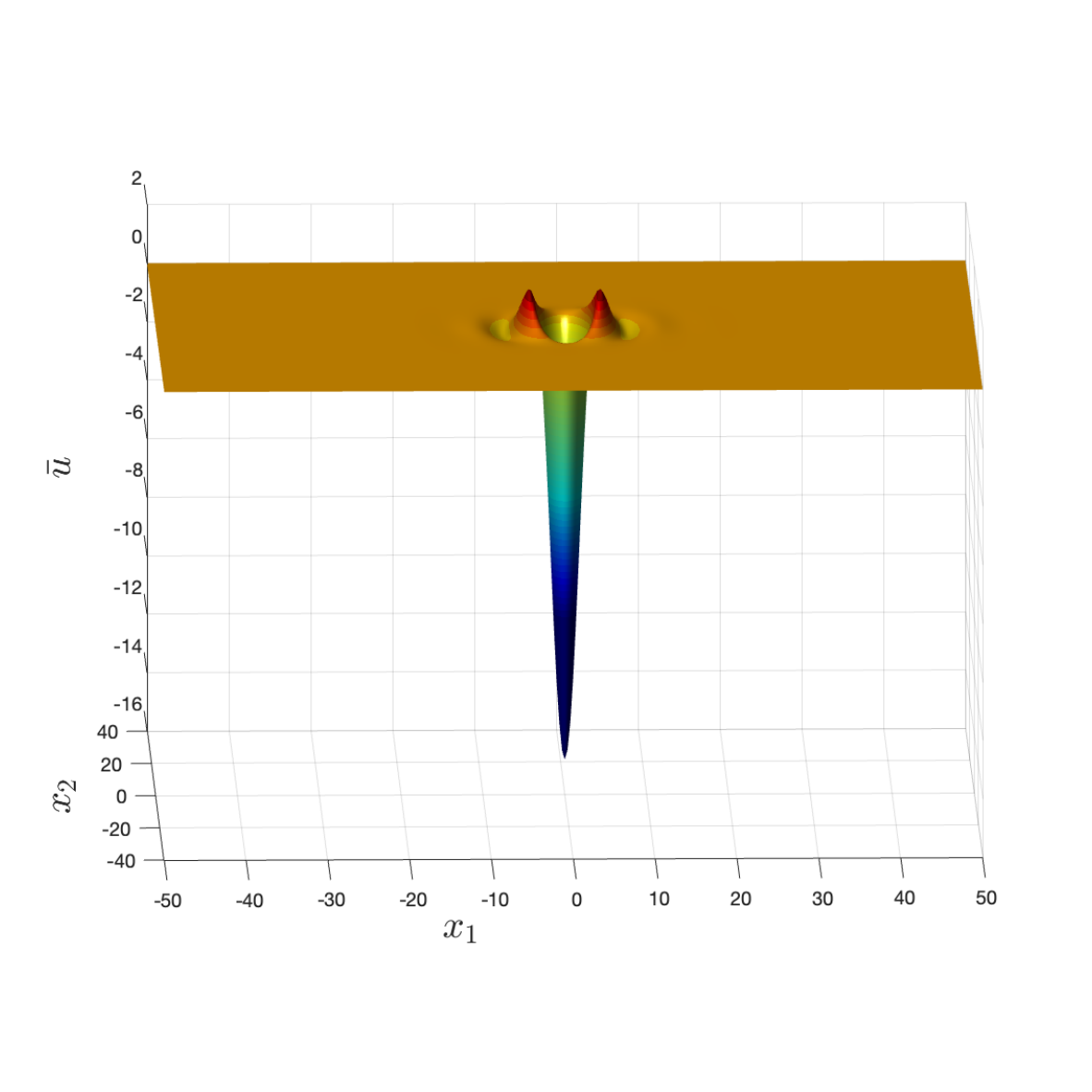}
        \vspace*{-15mm}
    \caption{Plot of the approximate solution $\bar{u}$ to \eqref{eq:SBeq} with boundary conditions \eqref{eq:SBBC}, justified in Theorem \ref{thm:examplesol}. The approximation is extended to the domain $[-L_1,L_1] \times [-L_2,L_2]$ in this plot, i.e., we visualize a single period.}
    \label{fig:examplesol}
\end{figure}

We argue that, by comparing the size of the error bound in Theorem~\ref{thm:examplesol} with (the amplitude of) the graph of $\bar{u}$ in Figure~\ref{fig:examplesol}, we can convincingly claim that the solution $\hat{u}$ describes a traveling pattern that localized in space on a fundamental domain (as well as periodic in the $X_1$-direction). Concerning regularity of the solution, although Theorem~\ref{thm:examplesol} provides an error estimate in terms of the supremum norm only, it follows immediately from the proof method that the solution $\hat{u}$ is real-analytic.

We prove this theorem using a computer-assisted proof (CAP). The  use of CAPs for rigorously verifying numerical simulations of dynamical systems has seen tremendous growth during the last decades. To ensure rigor of the computer calculations that are part of the proof, interval arithmetic is used (see e.g.~\cite{moore}). Examples of existence theorems for solutions to elliptic PDEs using computer-assisted means can be found in \cite{Nakao,Plum}, which mark some of the earliest instances of CAPs for PDEs. The developed techniques were part of the foundation for later studies on CAPs for PDEs described in \cite{Breuer2000,Oishi,Takayasu}. For periodic patterns such as the ones studied in this paper, Fourier transformations are a natural tool. This bring us in a setting where a Newton-Kantorovich type theorem can be applied, which is a standard tool when constructing CAPs for PDEs and beyond (e.g.~\cite{intro_breden2,intro_breden,intro_hungria,intro_reinhardt,intro_berg2,intro_berg3}). One of the Newton-Kantorovich approaches is known as the radii polynomial approach \cite{radii_Day}, which we follow in this paper. We note that a different powerful computer-assisted technique for finding periodic orbits in time is to use rigorous integration of the infinite dimensional flow combined with a topological argument (e.g.~\cite{ArioliKoch, Piotr2004}). More extensive overviews of CAPs for PDEs can be found in the book \cite{BookOverviewCAP} and the survey article \cite{gomez_overview}.

This paper builds on the aforementioned results and introduces several novel aspects. Indeed, as far as we know our approach is the first one to enable proving the existence of periodically localized traveling waves in the two-dimensional suspension bridge equation, confirming numerical indications from simulations that go two decades back~\cite{SB_horak}. 
As already stated, there is a stark contrast with the results in one spatial dimension, where a multitude of techniques is available. 
We note in particular that the computer-assisted techniques for the one-dimensional case as presented in~\cite{BergBreden2018} are not readily extendable to two dimensions. Namely, the method in ~\cite{BergBreden2018} is based on
introducing the variable $v=e^u$, so that $v$ and $u$ and its derivatives satisfy a system of polynomial ODEs, for which powerful CAP techniques are available.
However, in two dimensions there is, to our knowledge, no convenient analogue of this nonlinear change of variables.

To overcome this obstacle, we extend a recent method outlined in \cite{JP_FFT} to higher dimensions. This method is based on the techniques introduced in~\cite{Figueras} for controlling the aliasing error.
We present the method in this paper in two dimensions in such a way that the analogue for dimensions three and higher is also clear, essentially without any additional difficulties. In this paper, we limit ourselves to an exponential nonlinearity, but the method can be applied to other analytic nonlinearities, polynomial or nonpolynomial, as well, and the presentation makes this apparent (see also~\cite{JP_FFT}).
We hasten to add that there are certainly alternatives to our approach. In particular, Example 9.7 in \cite{BookOverviewCAP}, which is based on \cite{Plum_1995}, considers the Gelfand equation, a second-order elliptic PDE with an exponential nonlinearity, on a two dimensional domain. Using computer-assisted means and a finite element approach, smooth solution branches are proven for  that problem. We note that the same book also covers the suspension bridge equation (based on the work in \cite{wave_breuer}), but only in one spatial dimension.

Finally we mention a more technical point. 
Any CAP combines human effort with computations delegated to a computer. In an infinite dimensional setting this usually means that the computations are done on a finite dimensional approximation of the full problem, and the truncation dimension is one of the crucial parameters in such a method. To deal with the resulting ``projection'' or ``truncation'' error, complementary estimates are derived using pen-and-paper analysis.
In our construction, there are multiple steps where we need to choose the dimension of a Galerkin projection determining the amount of modes of a Fourier series that are represented in the computer. 
The projection dimensions should be chosen judiciously and do not need to be the same in each step of the proof.
We highlight the independence of these choices in this paper. The main advantage of this distinction is that by choosing the parameters independently, we create more flexibility, which we can exploit to decrease computation time as well as memory requirements, hence bringing a proof within practical range.
This is especially important in higher dimensions.

The techniques outlined in this paper open the door for future work: they can be used for studying PDEs in higher dimensions with general analytic nonlinear terms (that may be polynomial or not). Furthermore, our results provide the first proof of periodic localized solutions of the two-dimensional suspension bridge equation. These may serve as a starting point for proving existence of solutions that decay to $0$ as $X_1 \to \pm \infty$
by combining our techniques with those developed recently in \cite{Cadiot} and \cite{Cadiot2}. Another follow-up study would be the investigation of (in)stability of the solutions, as was done for the one-dimensional case in \cite{SB_nagatou}.

The rest of the paper is organized as follows. In Section \ref{sec:setup}, we introduce notation that will be convenient later on. Furthermore, we transform our problem to a fixed-point problem which prepares us for applying the Newton-Kantorovich type existence Theorem~\ref{thm:radii}.
In Section \ref{sec:auxestimates}, we describe how to rigorously compute intervals containing the Fourier coefficients corresponding to the nonlinear term of \eqref{eq:SBeq}. These intervals are essential for computing the bounds necessary to apply the existence Theorem~\ref{thm:radii}. Explicit expressions for these bounds are derived in Section~\ref{sec:bounds}. Subsequently, the proof of Theorem~\ref{thm:examplesol} is presented in Section~\ref{sec:results}. We also explore some other solutions, see Figure~\ref{fig:overviewplots}. In the closing section \ref{sec:powerseries}, we outline and discuss an alternative computer-assisted method for proving solutions to the suspension bridge equation based on a power series expansion of the nonlinearity. Finally, in Appendix \ref{app:overviewN} an overview of the different projection dimensions is presented, in Appendix \ref{app:calculationC} we describe how to establish in practice a rigorous computational bound on certain constant that appears in the estimates, and in Appendix \ref{app:poisson} the theorem and short proof of the two-dimensional discrete Poisson summation formula are presented.

\section{Setup and notation}\label{sec:setup}
We start by transforming \eqref{eq:SBeq} with boundary conditions \eqref{eq:SBBC} to a zero-finding problem in Fourier space. To shorten notation, we introduce frequency parameters $q_1,q_2>0$ such that $L_1= \frac{\pi}{q_1}$ and $L_2= \frac{\pi}{q_2}$. For $a_{|n|}=a_{|n_1|,|n_2|}=a_{n_1,n_2}=a_n$, we use the Fourier series expansion
\begin{align}\label{eq:defu2D}
   u(x_1,x_2) =\sum_{n \in \mathbb{Z}^2} a_{|n|} e^{i (n_1q_1x_1+n_2q_2x_2)}. 
\end{align}

\begin{remark}
The absolute values in \eqref{eq:defu2D} appear because of our focus on symmetric solutions. Hence, we can also write $u$ as the cosine series
\begin{align*}
    u= \sum_{n\in \mathbb{N}^2}\multi_n a_n\cos(n_1q_1x_1)\cos(n_2q_2x_2) ,
\end{align*}
where
\begin{align}\label{eq:defgn}
    \multi_n=\multi_{n_1,n_2}:=\begin{cases}
    1 & \text{ for } n_1=n_2=0,\\
      2 & \text{ for } n_1=0,n_2>0,\\
        2 &\text{ for } n_1>0,n_2=0,\\
          4 &\text{ for } n_1>0,n_2>0.
    \end{cases}
\end{align}
\end{remark}

Let $\mathcal{G}$ denote the function describing the nonlinear part of \eqref{eq:SBeq}, i.e., $\mathcal{G}(u)=e^u-u-1$, and let $a=\left(a_n\right)_{n \in \mathbb{Z}^2}$. Furthermore, let $G_n(a) =  \left(G(a)\right)_n=\left(G(a)\right)_{n_1,n_2}=\left(G(a)\right)_{|n_1|,|n_2|}=\left(G(a)\right)_{|n|}$ 
denote the Fourier coefficients of $\mathcal{G}\circ u$. We will not always explicitly mention the dependence of $G_n$ on $a$. The relation between $G_n$ and $a_n$ is defined via
\begin{align*}
   \mathcal{G}\left(\sum_{n \in \mathbb{Z}^2} a_{|n|} e^{i (n_1q_1x_1+n_2q_2x_2)}\right)=\sum_{n \in \mathbb{Z}^2} G_{|n|} e^{i (n_1q_1x_1+n_2q_2x_2)}.
\end{align*}
We will use similar notation for the derivative compositions $\mathcal{G}'\circ u$ and $\mathcal{G}''\circ u$.

To represent the linear part of~\eqref{eq:SBeq} in Fourier space, we introduce
\begin{align}\label{eq:deflambda}
\lambda_n=\lambda_{n_1,n_2}:=\left(n_1^2q_1^2+n_2^2q_2^2 \right)^2-c^2 n_1^{2}q_1^2+1.
\end{align}

\begin{remark}
In this splitting into a linear and nonlinear part we have made the choice to use the linearization of~\eqref{eq:SBeq} around $0$ as the linear part.
We also explored the alternative where the nonlinearity is $\widehat{\mathcal{G}}(u)=e^u-1$ with corresponding $\widehat{\lambda}_n=\left(n_1^2q_1^2+n_2^2q_2^2 \right)^2-c^2 n_1^{2}q_1^2$. This did not significantly influence the applicability of our method.
\end{remark}

Plugging the Fourier series expansion, as given in \eqref{eq:defu2D}, into \eqref{eq:SBeq} leads to the zero-finding problem $F(a)=0$ where
$$
F_{n}(a) :=\lambda_n a_{n}+\left(G(a)\right)_n \qquad \text{for all } n \in \mathbb{N}^2.
$$ 
To keep track of finite truncations and corresponding (Galerkin) projections, for any $N \in \mathbb{N}^2$ we introduce the index sets
\begin{align*}
I_{N}&:=\{n\in \mathbb{Z}^2 : |n_1|\leq N_1,|n_2|\leq N_2\},\\
I^+_{N}&:= I_N \cap \mathbb{N}^2 = \{0\leq n_1\leq N_1,0\leq n_2 \leq N_2\}.
\end{align*} 
To shorten notation, when we write $n\notin I_N^+$ we mean $n \in \mathbb{N}^2 \setminus I^+_N$, and when we write $n\notin I_N$ we mean $n \in \mathbb{Z}^2 \setminus I_N$.

The amount of (nonzero) Fourier coefficients used in the numerical approximation of a solution is denoted by $N^\mathrm{Gal}\in \mathbb{N}^2$. The size of the numerically approximated  part of the Fr\'echet derivative of $F$ is denoted by $N^\mathrm{Jac}\in \mathbb{N}^2$. An overview of the different computational truncation parameters $N$ used in this paper can be found in Appendix~\ref{app:overviewN}. 




\begin{remark}\label{rem:lambda}
Later on we work with the reciprocal of $\lambda_n$. In particular, we will need that $\lambda_n \neq 0$ for $n \notin I^+_{N^\mathrm{Jac}}$. For $c^2 < 2$ (which in practice is always the case) this is automatic, while for $c^2 \geq 2$ an elementary calculation shows that we require $(N^\mathrm{Jac}_1+1)^2 q_1^2 > \frac{c^2}{2}+[\frac{c^4}{4}-1]^{1/2}$ and $(N^\mathrm{Jac}_2+1)^2 q_2^2 > \frac{c^4-4}{4c^2}$.  

Furthermore, to keep the expressions simple, in our estimates we will assume that  $\lambda_n$ increases as $n$ increases for $n \notin I^+_{N^\mathrm{Jac}}$, see Lemma~\ref{lemma:lambda}. Therefore, we choose $N^\mathrm{Jac}$ such that $(N^\mathrm{Jac}_1+1)^2q_1^2 \geq \frac{c}{2}$. In summary, we thus require
\begin{equation}\label{e:Njacmin}
  N^\mathrm{Jac}_1 > \frac{[c^2+(\max\{0,c^4-4\})^{1/2}]^{1/2}}{2^{1/2} q_1}-1 \qquad\text{and}\qquad N^\mathrm{Jac}_2 > \frac{(\max\{0,c^4-4\})^{1/2}}{2cq_2}-1.
\end{equation}
\end{remark}

Let $\chi=\mathbb{R}^{\mathbb{N}^2}$ be the space of Fourier-cosine coefficients. Elements of $\chi$ are usually denoted by $a=\{a_n\}_{n\in \mathbb{N}^2}$.
For $N\in \mathbb{N}^2$, the (Galerkin) projection $\Pi^N$ of $\chi$ onto the  $(N_1 +1)(N_2+1)$-dimensional subspace $\pi^N \chi$ of $\chi$ is given by
\begin{align*}
    \left(\Pi^N a\right)_n = \begin{cases}
        a_n & \text{ if } n \in  I^+_N,\\
        0 & \text{ if } n \notin I^+_N.
    \end{cases}
\end{align*}
We can identify $\Pi^N \chi$ with $\mathbb{R}^{(N_1 +1)(N_2+1)}$ and linear operators on $\Pi^N \chi$ can be represented by matrices of size $(N_1 +1)(N_2+1) \times  (N_1 +1)(N_2+1)$.

Let $\bar{a} \in \pi^{N^\mathrm{Gal}}\chi$ be a numerical approximate zero of $F$, i.e., $\bar{a}_n=0$ for $n \notin I_{N^\mathrm{Gal}}$ and  $\Pi^{N^\mathrm{Gal}}F(\bar{a})\approx 0$.
For the Newton-Kantorovich method we will need an approximation $A$ of the inverse of the Fr\'echet derivative of $F$ at $\bar{a}$. To construct such an approximation, we consider the Jacobian matrix associated to the restricted map
$F^{N^\mathrm{Jac}} := \left.\pi^{N^\mathrm{Jac}}  F \right|_{\pi^{N^\mathrm{Jac}}\chi}$ evaluated at $\pi^{N^\mathrm{Jac}} \bar{a}$. Here $N^\mathrm{Jac}$ is not necessarily equal to $N^\mathrm{Gal}$. We denote by $A^{N^\mathrm{Jac}}$ the 
$( N^\mathrm{Jac}_1+1)  ( N^\mathrm{Jac}_2+1)\times 
( N^\mathrm{Jac}_1+1) ( N^\mathrm{Jac}_2+1)$ matrix obtained by numerically computing an approximate inverse of the Jacobian matrix
$DF^{N^\mathrm{Jac}} (\Pi^{N^\mathrm{Jac}}\bar{a})$.
Using $A^{N^\mathrm{Jac}}$ we can define the operator $A$ as
\begin{equation}\label{eq:defA}
    \left(Av\right)_n:= \begin{cases} (A^{N^\mathrm{Jac}}\Pi^{N^\mathrm{Jac}}v)_n & \text{ if } n \in I^+_{N^\mathrm{Jac}},\\
   \lambda^{-1}_n v_n & \text{ if }  n \notin I^+_{N^\mathrm{Jac}}.
    \end{cases}
\end{equation}
 
\begin{remark}
When evaluating $\Pi^{N^\mathrm{Gal}}F(\bar{a})$ and $DF^{N^\mathrm{Jac}} (\Pi^{N^\mathrm{Jac}}\bar{a})$ we will encounter $\Pi^{N^\mathrm{Gal}} G(\bar{a})$ and $\Pi^{N^\mathrm{Jac}}G'(\bar{a})$.
While we cannot compute these terms exactly, at this stage of the construction that is not a problem, since we only need numerical approximations. In Section~\ref{sec:auxestimates} we delve deeper into obtaining rigorous enclosures of these terms.
\end{remark}

\begin{remark}\label{rem:Njac}
For simplicity the two projection dimensions are usually chosen to be equal: $N^\mathrm{Gal}=N^\mathrm{Jac}$. In that ``natural'' case the Jacobian matrix used for the Newton iterations that produce $\bar{a}$ is the same as the one inverted to obtain $A^{N^\mathrm{Jac}}$. However, when trying to obtain a proof it is often only necessary to enlarge \emph{either} $N^\mathrm{Gal}$ \emph{or} $N^\mathrm{Jac}$ to have a successful CAP. Hence distinguishing between $N^\mathrm{Gal}$ and $N^\mathrm{Jac}$ allows us to minimize the computation time and memory usage, which is why we make the distinction explicit. 
\end{remark}

To turn a subset of the coefficient space $\chi$ into a Banach space, we choose a weighted $\ell^1$ space, where the weights are represented by $\nu=(\nu_1,\nu_2) \geq (1,1)$. Here and in what follows, for $x,\tilde{x}\in \mathbb{R}^2$ we write $x \geq \tilde{x}$ to say that both $x_1 \geq \tilde{x}_1$ and $x_2 \geq \tilde{x}_2$. Similarly, $x > \tilde{x}$ means that both $x_1 > \tilde{x}_1$ and $x_2 > \tilde{x}_2$.

We now denote by $\ell^1_\nu$ the Banach space of elements $a \in \chi$ with finite norm
\begin{align*}
\|a\|_{\nu} := \sum_{n \in \mathbb{N}^2}\left|a_{n}\right| \omega_{n} ,
\end{align*}
with weights
\begin{align*}
 \omega_n :=  \multi_n\nu^{n} := \multi_n \nu_1^{n_1}\nu_2^{n_2} ,
\end{align*}
where we recall that $\multi_n$ is given by \eqref{eq:defgn}.
The weights $\omega_n$ are such that we can also write 
\begin{align*}
    \|a\|_{\nu} = \sum_{n \in \mathbb{Z}^2}\left|a_{|n|}\right| \nu^{|n|},
\end{align*}
where $|n|=(|n_1|,|n_2|)$. We note that for $a \in \ell^1_\nu$ with $\nu \geq 1$ the Fourier series~\eqref{eq:defu2D} converges uniformly, while for $\nu>1$ the series represents a smooth (real analytic) function. 
One of the other benefits of this space is that the norm satisfies the Banach algebra property, i.e.,
\begin{align}\label{eq:banachalgebra}
    \|a\ast \tilde{a}\|_\nu \leq \|a\|_\nu \|\tilde{a}\|_\nu ,
	\qquad\text{for } a,\tilde{a} \in \ell_{\nu}^{1},
\end{align}
where the convolution product is defined as
\begin{align*}
    (a\ast \tilde{a})_n=\sum_{m\in \mathbb{Z}^2}a_{|m|}\tilde{a}_{|n-m|}
	\qquad\text{for } n \in \mathbb{N}^2.
\end{align*} 
Finally, let $I$ denote the identity map (its domain should be clear from context).
We are now ready to state the existence theorem used for obtaining the CAPs in this paper, which is based on establishing that $a \mapsto a - AF(a)$ is a contraction on a ball of appropriately chosen radius centered at $\bar{a}$.

\begin{theorem}[e.g.~\cite{Breden_gPC}]\label{thm:radii}
  Let $F: X \rightarrow \hat{X}$ be a Fr\'echet differentiable map between Banach spaces $X$ and $\hat{X}$. Furthermore, let $A \in B(\hat{X}, X)$ be injective, $\bar{a}\in X$ and $r^\ast \in (0,\infty]$. Assume that $Y,Z, W \in \mathbb{R}^+$ satisfy the bounds
\begin{align}
\|A F(\bar{a})\|_{X} & \leq Y, \label{eq:Ybound}\\
\left\|I-A DF(\bar{a})\right\|_{B(X)} & \leq Z, \label{eq:Zbound}\\
\|A[DF(\bar{a}+w)-DF(\bar{a})]\|_{B(X)}&\leq W\|w\|_X \quad \text{ for all } \|w\|_X\leq r^\ast.\label{eq:Wbound}
\end{align}
If $Y,Z$ and $W$ satisfy
\begin{align*}
    Z&<1;\\
    2YW&<(1-Z)^2,
\end{align*}
then, for any $r$ satisfying
\begin{align*}
    \frac{1-Z-\sqrt{(1-Z)^2-2YW}}{W}\leq r < \min{\left(\frac{1-Z}{W},r^\ast\right)}
\end{align*}
there exists a unique $\hat{a} \in X$ satisfying $F(\hat{a})=0$ such that 
$\|\hat{a}-\bar{a}\|_X \leq r$.
\end{theorem}

When we apply this theorem in our setting we take $X=\ell^1_\nu$ and we can for example select $\hat{X}= \{a \in \chi : \|a\|_{\hat{X}} = \sum_{n\in\mathbb{N}^2} (|\lambda_n|+1)^{-1} |a_n| \omega_n <\infty \}$, so that $A: \hat{X} \rightarrow X$ is a bounded operator. Concerning injectivity of $A$, it follows from~\eqref{eq:defA} that it suffices to check that the matrix $A^{N^\mathrm{Jac}}$ is invertible, which is not difficult to verify (for our norm it is implied by the requirement that $\|I-ADF(\bar{a}) \|_{B(X)}<1$).

\section{Determining the Fourier coefficients of the nonlinear term}\label{sec:auxestimates}

Assuming that we know $\bar{a}$, which has only finitely many nonzero terms and represents a numerical approximate zero of $F$, we want to calculate corresponding $\bar{b}=G(\bar{a})$, the Fourier coefficients of $\mathcal{G} \circ \bar{u}$ where $\bar{u}$ is the function with Fourier coefficients $\bar{a}$. However, we cannot calculate $\bar{b}$ exactly. One reason for this is that $\bar{b}$ has infinitely many non-vanishing Fourier coefficients, since $\mathcal{G}$ is not polynomial. But the situation is even worse, since for nonpolynomial $\mathcal{G}$ not even finitely many of the coefficients of $\bar{b}$ can be calculated exactly. Hence, in this section we perform estimates to obtain intervals containing the values of the Fourier coefficients. To do so, we follow~\cite{Figueras} and extend the one-dimensional method outlined in \cite{JP_FFT} to higher dimensions. To be specific, we focus on \emph{two} dimensions, but our notation is chosen such that the generalization to dimensions three and higher is straightforward. We present the estimates for general Fourier series instead of cosine series. However, in the implementation in the code we do use the symmetries in our problem and hence we work with cosine series there. We also note that the implementation uses interval arithmetic to rigorously absorb errors due to rounding in floating point operations.

The enclosure we obtain for $\bar{b}$ is described in two steps. First, in Section~\ref{sec:generalbound} we derive a relatively rough bound on $|\bar{b}_n|$ that holds for all $n \in \mathbb{Z}^2$. While a bit rough, the bound decays exponentially in $n$, and in practice we use this bound for large $n$. 
Subsequently, we construct an improved enclosure of $\bar{b}_n$ for small $n \in \mathbb{Z}^2$ in Section~\ref{sec:alias}. Here the boundary between small and large $n$ is determined by a computational parameter $N^{\mathrm{Alias}} \in \mathbb{N}^2$.  We choose $N^{\mathrm{Alias}}\geq N^{\mathrm{Jac}}$ for implementation purposes. Combining the two results, we arrive at a way of rigorously computing intervals containing $\bar{b}_n$.

\subsection{Enclosure for the tail}\label{sec:generalbound}

Since $\bar{a} \in \chi$ has only finitely many nonvanishing terms, 
the corresponding function
$\bar{u}(z) = \sum_{n\in \mathbb{Z}^2} \bar{a}_n e^{i(n_1 z_1+n_2 z_2)}$ is analytic on $\mathbb{C}^2$. 
Note that this function is $2\pi$-periodic in the real parts of $z_1$ and $z_2$, hence we have absorbed the frequencies $(q_1,q_2)$ in the independent variables $z=(z_1,z_2)$.
For $\bar{\rho} \in \mathbb{R}^2$ we define
\begin{align}\label{eq:defCrho}
C_{\bar{\rho}}:=\frac{1}{(2 \pi)^2} \int_{-\pi}^{\pi} \int_{-\pi}^{\pi}|\mathcal{G}(\bar{u}(x_1-i \bar{\rho}_1,x_2-i \bar{\rho}_2))| d x_1dx_2 .
\end{align}  
Furthermore, for $\bar{\rho} > 0$ we set
\begin{align}\label{eq:defC}
\widehat{C}_{\bar{\rho}} := \max \left\{C_{\bar{\rho}_1,\bar{\rho}_2}, C_{-\bar{\rho}_1,\bar{\rho}_2}, C_{\bar{\rho}_1,-\bar{\rho}_2}, C_{-\bar{\rho}_1,-\bar{\rho}_2}\right\}.
\end{align}

\begin{remark}
The constant $\widehat{C}_{\bar{\rho}}$ defined in \eqref{eq:defC} cannot be calculated exactly. However, we can rigorously and efficiently estimate it using the discrete Fourier transform and interval arithmetic. This procedure is described in Appendix \ref{app:calculationC}.
\end{remark}

The following lemma controls the decay rate of the Fourier coefficients of $\mathcal{G}(\bar{u})$. The result holds more generally for any function $\bar{u}$ which is analytic on some strip 
\[
S_\rho := \{ z \in \mathbb{C}^2 : \mathrm{Im}(z_1) < \rho_1 , \mathrm{Im}(z_2) < \rho_2 \}
\]
of width $\rho>\bar{\rho}$ and any nonlinearity $\mathcal{G}$ which is analytic on the range of $\left.\bar{u}\right|_{S_\rho}$.

\begin{lemma}\label{lemma:boundb} 
Let $\mathbb{R}^2 \ni \bar{\rho} >0$.
 Let $\{\bar{b}_n\}_{n\in \mathbb{Z}^2}$ denote the Fourier coefficients of $\mathcal{G}\circ \bar{u}$ and 
let $\bar{\nu} = (e^{\bar{\rho}_1}, e^{\bar{\rho}_2})$. Then     \begin{align*}
\left|\bar{b}_{n}\right| \leq \frac{\widehat{C}_{\bar{\rho}}}{\bar{\nu}^{|n|}}, \quad \text { for all } n \in \mathbb{Z}^2.
\end{align*}
\end{lemma}

\begin{proof}
    We present the details for $n \geq 0$,  which corresponds to $C_{\bar{\rho}}$. The other cases ($n < 0$ and $n_1< 0,n_2 \geq 0$ and $n_1\geq 0,n_2 < 0$) are dealt with in an equivalent way. Choose any $\rho>\bar{\rho}$. As $\bar{u}$ is analytic on $\mathbb{C}^2$, both $\bar{u}$ and $\mathcal{G} \circ \bar{u}$ are analytic on the strip $S_\rho$. Hence, for the closed contours $\Gamma_j$, $j=1,2$ depicted in Figure~\ref{fig:AnalyticStrip} we have $\iint_{\Gamma_1 \times \Gamma_2} \mathcal{G}(\bar{u}(z)) dz=0$ by Cauchy's Integral Theorem.
\begin{figure}[t]
    \centering
    \includegraphics[width=0.65\textwidth]{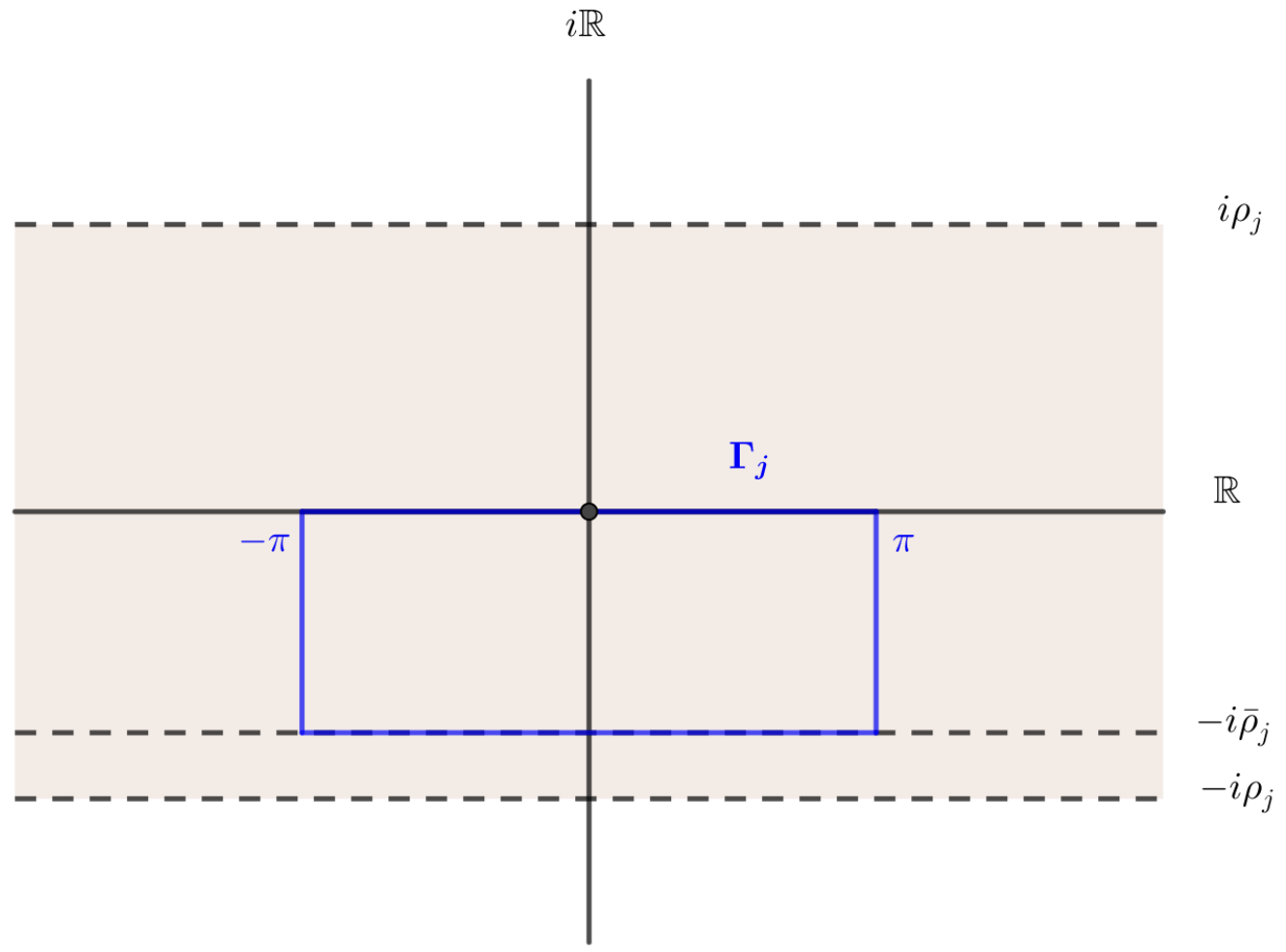}
    \caption{The placement of the strip and the contours used in the proof of Lemma~\ref{lemma:boundb}.}
    \label{fig:AnalyticStrip}
\end{figure}
By using periodicity of $\bar{u}$ in $\mathrm{Re}(z_j)$, we rewrite the formula for the Fourier coefficients $\bar{b}_n$ as
$$
\begin{aligned}
\bar{b}_{n} & =\frac{1}{(2 \pi)^2} \int_{-\pi}^{\pi} \int_{-\pi}^{\pi}\mathcal{G}(\bar{u}(x_1,x_2)) e^{-i (n_1 x_1+n_2 x_2)} d x_1dx_2 \\
& =\frac{1}{(2\pi)^2} \int_{-\pi}^{\pi} \int_{-\pi}^{\pi} \mathcal{G}(\bar{u}(x_1-i \bar{\rho}_1,x_2-i \bar{\rho}_2)) e^{-i (n_1 x_1+n_2 x_2)-(n_1 \bar{\rho}_1+n_2 \bar{\rho}_2)} d x_1dx_2 \\
& =\frac{1}{(2 \pi)^2} e^{-(n_1\bar{\rho}_1+n_2\bar{\rho}_2)} \int_{-\pi}^{\pi} \int_{-\pi}^{\pi} \mathcal{G}(\bar{u}(x_1-i \bar{\rho}_1,x_2-i \bar{\rho}_2)) e^{-i (n_1 x_1+n_2 x_2)} d x_1dx_2.
\end{aligned}
$$ 
Consequently, we obtain, for any $n \geq 0$,
$$
\begin{aligned}
\left|\bar{b}_{n}\right| & \leq \frac{1}{(2 \pi)^2 e^{n_1\bar{\rho}_1+n_2\bar{\rho}_2}} \int_{-\pi}^{\pi} \int_{-\pi}^{\pi}\left|\mathcal{G}(\bar{u}(x_1-i \bar{\rho}_1,x_2-i \bar{\rho}_2))\right|\left|e^{-i (n_1 x_1+n_2 x_2)}\right| d x_1dx_2 \\
& = \frac{1}{\bar{\nu}_1^{n_1}\bar{\nu}_2^{n_2}}\frac{1}{(2 \pi)^2} \int_{-\pi}^{\pi} \int_{-\pi}^{\pi}|\mathcal{G}(\bar{u}(x_1-i \bar{\rho}_1,x_2-i \bar{\rho}_2))| d x_1dx_2= \frac{C_{\bar{\rho}}}{\bar{\nu}^{n}}.
\end{aligned}
$$
Adjusting the contour to the upper half plane where appropriate, we gather the constants for the other cases, e.g.
\begin{align*}
\left|\bar{b}_{n}\right| \leq \frac{C_{-\bar{\rho}_1,\bar{\rho}_2}}{\bar{\nu}^{|n|}} \quad\text{ for all } n_1<0,n_2\geq0.
\end{align*}
Combining the different cases and using the definition of $\widehat{C}_{\bar{\rho}}$ in \eqref{eq:defC} proves the result.
\end{proof}

\begin{remark}\label{remark:Crhoforcosines}
In the case of cosine series we have $\bar{b}_n=\bar{b}_{|n|}$. We are thus only interested in $n \in \mathbb{N}^2$ and it easily follows, for example from the proof above, that $\widehat{C}_{\bar{\rho}}=C_{\bar{\rho}}$ for $\bar{\rho} > 0$.
\end{remark}

\subsection{Aliasing error}\label{sec:alias}

In this section we aim to find tight enclosures on $\bar{b}_n$ for $n\in I_{N^\mathrm{Alias}}$. The choice of the computational parameter $N^\mathrm{Alias}$ may depend on the purpose of our calculations, but will in our case be guided by the bounds we need to compute in Section~\ref{sec:bounds}. We thus fix $N^\mathrm{Alias}$ and set out to improve the general bound from Lemma~\ref{lemma:boundb} for $n \in I_{N^\mathrm{Alias}}$. Namely, we use the \emph{discrete} Fourier transform to calculate a numerical approximation of $\bar{b}$, which we call $\bar{b}^\mathrm{FFT}$, and bound the resulting \emph{aliasing} error via the Poisson summation formula, as explained below.

In order to calculate $\bar{b}^\mathrm{FFT}$, we introduce $N^\mathrm{FFT}$ indicating how many Fourier modes are considered when doing discrete (inverse) Fourier transformations. For algorithmic reasons, we choose $N_1^\mathrm{FFT}$ and $N_2^\mathrm{FFT}$ to be powers of $2$, but for the analysis that is irrelevant. In order to obtain a small aliasing error for $n \in I_{N^\mathrm{Alias}}$, 
we require that $N^\mathrm{FFT} \gg N^\mathrm{Alias}$, see also Remark~\ref{rem:sizeNFFT}.

We introduce the index set
\begin{align*}
J_{N^\mathrm{FFT}} := \{-N_1^{\mathrm{FFT}}\leq n_1 \leq N_1^{\mathrm{FFT}}-1,-N_2^{\mathrm{FFT}}\leq n_2 \leq N_2^{\mathrm{FFT}}-1\},
\end{align*}
and define the uniform mesh  
(of the square $[-\pi,\pi] \times [-\pi,\pi]$)
\begin{align}\label{eq:mesh}
    x_{1,k_1}=\frac{\pi k_1}{N_1^{\mathrm{FFT}}}, \quad   x_{2,k_2}=\frac{\pi k_2}{N_2^{\mathrm{FFT}}} \qquad\qquad \text{ for } k \in J_{N^\mathrm{FFT}}.
\end{align}
Then $\bar{b}^\mathrm{FFT}$ is defined via
\begin{align}\label{eq:defbbar1}
\bar{b}^\mathrm{FFT}_{n} :=\frac{1}{4N_1^{\mathrm{FFT}}N_2^{\mathrm{FFT}}} \sum_{k \in J_{N^\mathrm{FFT}}}\mathcal{G}\left(\bar{u}\left(x_{1,k_1},x_{2,k_2}\right)\right) e^{-i (n_1 x_{1,k_1}+n_2 x_{2,k_2})}.
\end{align}
When $\bar{u}$ has only finitely many nonvanishing Fourier coefficients, then the right-hand side of~\eqref{eq:defbbar1} can be evaluated rigorously via interval arithmetic for all $n \in J_{N^\mathrm{FFT}}$.
Before stating the bound on the aliasing error, 
we introduce 
\begin{align*}
    f_\mathrm{geo}(\xi,N) :=\;&
 \frac{(1+\xi_1)(1+\xi_2) - (1+\xi_1-2\xi_1^{N_1+1})
      (1+\xi_2-2\xi_2^{N_1+1})}{(1-\xi_1)(1-\xi_2)}
    \\[1mm]
    =\;& \frac{
    2\xi_1^{N_1+1}\left(1+\xi_2\right)+
    2\xi_2^{N_2+1}\left(1+\xi_1\right)
    -4\xi_1^{N_1+1}\xi_2^{N_2+1}
    }{\left(1-\xi_1\right)\left(1-\xi_2\right)},
\end{align*} 
and prove an auxiliary lemma.
\begin{lemma}\label{lemma:geometric}
 Let $N\in\mathbb{N}^2$ and let $\xi=(\xi_1,\xi_2)$ be such that $\xi_1,\xi_2 \in [0,1)$.
Then
\begin{align}\label{eq:geom}
   \sum_{n\notin I^+_N}\multi_n \xi^{|n|}= 
   \sum_{n\notin I_N}\xi^{|n|}=  
   f_\mathrm{geo}(\xi,N) . 
\end{align}
\end{lemma}
\begin{proof}
The first identity in~\eqref{eq:geom} follows directly from the definition of $\multi_n$ in~\eqref{eq:defgn}. To prove the second identity, we use the splitting
\begin{align}\label{eq:sumsplit}
    \sum_{n\notin I_N }\xi^{|n|}=\sum_{n \in \mathbb{Z}^2}\xi^{|n|}-\sum_{n\in I_N}\xi^{|n|}.
\end{align}
We start by rewriting the second summation as
\begin{align*}
    \sum_{n\in I_N}\xi^{|n|}
    &=\sum_{n_1=-N_1}^{N_1}\xi_1^{|n_1|}\sum_{n_2=-N_2}^{N_2}\xi_2^{|n_2|}=\left(2\sum_{n_1=0}^{N_1}\xi_1^{n_1}-1\right)\left(2\sum_{n_2=0}^{N_2}\xi_2^{n_2}-1\right)\\
    &=\left(\frac{1-2\xi_1^{N_1+1}+\xi_1}{1-\xi_1}\right)\left(\frac{1-2\xi_2^{N_2+1}+\xi_2}{1-\xi_2}\right).
\end{align*}
Using this and recalling that $(\xi_1,\xi_2)<(1,1)$, we can write the first summation in \eqref{eq:sumsplit} as
\begin{align*}
    \sum_{n\in \mathbb{Z}^2}\xi^{|n|}&=\lim_{(N_1,N_2)\to (\infty,\infty)}\sum_{n\in I_N}\xi^{|n|}=\left(\frac{1+\xi_1}{1-\xi_1}\right)\left(\frac{1+\xi_2}{1-\xi_2}\right).
\end{align*}
Subtracting the second term from the first one and rewriting yields the assertion.
\end{proof}

We now present an estimate for the aliasing error, along the lines of~\cite{Figueras}.
\begin{lemma}\label{lemma:alias}
Let $\bar{\rho}=(\bar{\rho}_1,\bar{\rho}_2)>0$ and let  $\widehat{C}_{\bar{\rho}}$ be defined in~\eqref{eq:defC} and $\bar{\nu}_j= e^{\bar{\rho}_j}$ for $j=1,2$. Then
  \begin{align}\label{eq:aliasingerror}
\left|\bar{b}_n-\bar{b}^\mathrm{FFT}_n\right|\leq 
\widehat{C}_{\bar{\rho}} \, \bar{\nu}^{|n|} \, 
f_\mathrm{geo}\bigl(\bar{\nu}^{-2N^{\mathrm{FFT}}},0\bigr)  ,
\qquad \text { for } n \in J_{N^\mathrm{FFT}}.
\end{align}
\end{lemma}
\begin{proof}
The two-dimensional discrete Poisson summation formula (see Appendix \ref{app:poisson}) implies that
\begin{align*}
\bar{b}^\mathrm{FFT}_{n}=\bar{b}_{n}+\sum_{\substack{j\in\mathbb{Z}^2 \\ (j_1,j_2)\neq (0,0)}}\bar{b}_{n+2j N^{\mathrm{FFT}}}, \qquad \text { for } n \in J_{N^\mathrm{FFT}}.
\end{align*}
For $n \in J_{N^\mathrm{FFT}}$, we then obtain
\begin{align*}
\left|\bar{b}_n-\bar{b}^\mathrm{FFT}_n\right| & \leq \sum_{\substack{j\in\mathbb{Z}^2 \\ (j_1,j_2)\neq (0,0)}}\left|\bar{b}_{n+2j N^{\mathrm{FFT}}}\right| \\
& \leq \widehat{C}_{\bar{\rho}} \sum_{\substack{j\in\mathbb{Z}^2 \\ (j_1,j_2)\neq (0,0)}}\frac{1}{\bar{\nu}^{|2j N^{\mathrm{FFT}}|-|n|}} \qquad \text{(Lemma \ref{lemma:boundb})} \\
&=\widehat{C}_{\bar{\rho}} \bar{\nu}^{|n|}\sum_{\substack{j\in\mathbb{Z}^2 \\ j \notin I_0}}\bigl(\bar{\nu}^{-2N^\mathrm{FFT}}\bigr)^{|j|}\\
&= \widehat{C}_{\bar{\rho}} \bar{\nu}^{|n|} 
f_\mathrm{geo}\bigl(\bar{\nu}^{-2N^{\mathrm{FFT}}},0\bigr) .
                 \qquad \text{(Lemma \ref{lemma:geometric})}
				 \qedhere
\end{align*}
\end{proof}
We now obtain enclosures on $\bar{b}_n$ for $n \in I_{N^\mathrm{Alias}}$.
Combining Lemma \ref{lemma:boundb}  and Lemma \ref{lemma:alias} and setting
\begin{align*}
\bar{\varepsilon}_n:= 
\bar{\nu}^{|n|}  f_\mathrm{geo}\left(
                   \bar{\nu}^{-2N^{\mathrm{FFT}}},0\right) 
=\frac{2\bar{\nu}^{|n|}\left({\bar{\nu}_1^{-2N_1^{\mathrm{FFT}}}}+{\bar{\nu}_2^{-2N_2^{\mathrm{FFT}}}}\right)}{\left(1-\bar{\nu}_1^{-2N_1^{\mathrm{FFT}}}\right)\left(1-\bar{\nu}_2^{-2N_2^{\mathrm{FFT}}}\right)},
\end{align*}
we can enclose $\bar{b}_n$  by intervals 
\begin{align}\label{eq:bncompute}
    \bar{b}_n \in \begin{cases}
  \bar{b}^\mathrm{FFT}_n+\widehat{C}_{\bar{\rho}} [-\bar{\varepsilon}_n,\bar{\varepsilon}_n] & n \in I_{N^\mathrm{Alias}},\\ [-\frac{\widehat{C}_{\bar{\rho}}}{\bar{\nu}^{|n|}},\frac{\widehat{C}_{\bar{\rho}}}{\bar{\nu}^{|n|}}] & n \notin I_{N^\mathrm{Alias}}.
    \end{cases}
\end{align}

\begin{remark}\label{rem:sizeNFFT}
   The size of the aliasing errors is thus measured by the product of $\widehat{C}_{\log \bar{\nu}}$ and 
\[ \frac{2\bar{\nu}_1^{N_1^\mathrm{Alias}}\bar{\nu}_2^{N_2^\mathrm{Alias}}
\left({\bar{\nu}_1^{-2N_1^{\mathrm{FFT}}}}+{\bar{\nu}_2^{-2N_2^{\mathrm{FFT}}}}\right)}{\left(1-\bar{\nu}_1^{-2N_1^{\mathrm{FFT}}}\right)\left(1-\bar{\nu}_2^{-2N_2^{\mathrm{FFT}}}\right)}
,
\]
which quantifies the requirement that $N^{\mathrm{FFT}} \gg N^{\mathrm{Alias}}$ in order to obtain tight bounds.
\end{remark}

\begin{remark}\label{remark:barbprime}
The bounds in Theorem \ref{thm:radii} do not only involve $F(\bar{a})$, but also its derivative. Hence, we also need the Fourier coefficients $\bar{b}'=G'(\bar{a})$ of $\mathcal{G}' \circ \bar{u}$. 
As before, the enclosing intervals are constructed using the numerical approximation obtained through the discrete Fourier transform:
\begin{align*}
\bar{b}^{'\mathrm{FFT}}_{n} :=\frac{1}{4N_1^{\mathrm{FFT}}N_2^{\mathrm{FFT}}} \sum_{k \in J_{N^\mathrm{FFT}}}\mathcal{G}'\left(\bar{u}\left(x_{1,k_1},x_{2,k_2}\right)\right) e^{-i (n_1 x_{1,k_1}+n_2 x_{2,k_2})}.
\end{align*}
Repeating the procedure for $\mathcal{G}'(u)=e^u-1$ instead of $\mathcal{G}(u)=e^u-u-1$, in particular calculating $\widehat{C}'_{\bar{\rho}}$ from~\eqref{eq:defC} with $\mathcal{G}$ replaced by $\mathcal{G}'$, we end up with
\begin{align}\label{eq:bnprimecompute}
    \bar{b}'_n \in \begin{cases}
        {\bar{b}}^{'\mathrm{FFT}}_n+\widehat{C}'_{\bar{\rho}}[-\bar{\varepsilon}_n,\bar{\varepsilon}_n] & n \in I_{N^\mathrm{Alias}},\\
        [-\frac{\widehat{C}'_{\bar{\rho}}}{\bar{\nu}^{|n|}},\frac{\widehat{C}'_{\bar{\rho}}}{\bar{\nu}^{|n|}}] & n \notin I_{N^\mathrm{Alias}}.
    \end{cases}
\end{align}
We note that one does not need to use the same $N^{\mathrm{Alias}}$ (and $N^{\mathrm{FFT}}$) for $\mathcal{G}'$ as for $\mathcal{G}$, although in our code we did.

Finally, in Section~\ref{sec:Wbound} we will use the analogous construction for 
$\mathcal{G}''(u)=e^u$ and $\bar{b}''=G''(\bar{a})$. 
\end{remark}

\section{Bounds}
\label{sec:bounds}

In this section, formulas for the bounds $Y$, $Z$ and $W$ that appear in Theorem \ref{thm:radii} are derived. All these formulas can be evaluated (or at least bounded) algorithmically using interval arithmetic.
We assume that
\begin{itemize}
\item
	we have fixed $N^{\mathrm{Gal}}$ and computed numerically the Fourier coefficients  $\bar{a}$ of an approximation of a (candidate) solution. 
\item
    we have selected $N^{\mathrm{Jac}}$ satisfying~\eqref{e:Njacmin} and we have determined numerically a matrix $A^{N^{\mathrm{Jac}}}$ so that the operator $A$ is defined.
\item
    we have chosen $\bar{\rho}$ and $N^{\mathrm{Alias}} \ge N^{\mathrm{Jac}}$ and $N^{\mathrm{FFT}} \gg N^{\mathrm{Alias}}$ as introduced in Section~\ref{sec:auxestimates}. 
\end{itemize}	

\subsection{Computing $Y$}\label{sec:Ybound}
The first bound appearing in Theorem \ref{thm:radii} is the $Y$-bound, which should satisfy $
\|A F(\bar{a})\|_{\nu} \leq Y,
$
as stated in \eqref{eq:Ybound}. We can rewrite $\|A F(\bar{a})\|_{\nu}$ as
\begin{align}\label{eq:AF}
    \|A F(\bar{a})\|_{\nu} = \sum_{n \in I^+_{N^\mathrm{Alias}}}\left|\left(AF(\bar{a})\right)_n\right| \omega_{n}+\sum_{ n \notin I^+_{N^\mathrm{Alias}}}\left|\left(AF(\bar{a})\right)_n\right| \omega_n,
\end{align}
where $N^{\mathrm{Alias}} \ge N^{\mathrm{Jac}}$ essentially determines where we switch from relatively tight computation of terms to relatively loose estimation of the tail (which is nevertheless small in size in practice).
We start with the first summation.

\subsubsection{$n\in I^+_{N^\mathrm{Alias}}$}

Using the structure of $A$ given in \eqref{eq:defA} and using \eqref{eq:bncompute},  we write 
\begin{align*}
\sum_{n \in I^+_{N^\mathrm{Alias}}}\left|\left(AF(\bar{a})\right)_n\right| \omega_{n}\leq & \sum_{n \in I_{N^\mathrm{Jac}}^+}\left|\left(A^{N^\mathrm{Jac}} F^{N^\mathrm{Jac}}(
\bar{a})\right)_{n}\right| \omega_{n} 
+\sum_{\substack{n\in I^+_{N^\mathrm{Alias}}\\n \notin I_{N^\mathrm{Jac}}^+}} \frac{\left|\bar{b}^\mathrm{FFT}_{n}\right|+\widehat{C}_{\bar{\rho}}\bar{\varepsilon}_{n}}{\lambda_n} \omega_{n}.
\end{align*}
Even though both terms contain finitely many terms, we are not able to compute
$F^{N^\mathrm{Jac}}(\bar{a})$ exactly as we have that  $F_n(\bar{a})=\lambda_n\bar{a}_n+\bar{b}_n$, where we cannot evaluate $\bar{b}_n=G_n(\bar{a})$ exactly. Therefore, we use interval arithmetic both to compensate for rounding errors and to compute intervals containing $\bar{b}_n$ given in the right-hand side of~\eqref{eq:bncompute}.
With this interpretation of computability of $F^{N^{\mathrm{Jac}}}(\bar{a})$,
we set  
\begin{align*}    
Y^{N^{\mathrm{Alias}}}:=\sum_{n \in I_{N^\mathrm{Jac}}^+}\left|\left(A^{N^\mathrm{Jac}} F^{N^\mathrm{Jac}}(\bar{a})\right)_{n}\right| \omega_{n}+\sum_{\substack{n\in I^+_{N^\mathrm{Alias}}\\n \notin I_{N^\mathrm{Jac}}^+}} \frac{\left|\lambda_n\bar{a}_n+\bar{b}^\mathrm{FFT}_{n}\right|+\widehat{C}_{\bar{\rho}}\bar{\varepsilon}_{n}}{\lambda_n} \omega_{n},
\end{align*}
which bounds the first summation in~\eqref{eq:AF}.

\subsubsection{$n \notin I^+_{N^\mathrm{Alias}}$}

We assume $N^{\mathrm{Jac}}$ satisfies the two inequalities in ~\eqref{e:Njacmin}.
Letting $\lceil \cdot \rceil$ denote the ceiling function, we introduce
\begin{align*}
    \lambda_{\mathrm{min}} (N) := \min\{\lambda_{N_1+1,0},\underset{0\leq i\leq \lceil \frac{c}{q_1}\rceil}{\min}\lambda_{i,N_2+1} \},
\end{align*} 
which is used in the next lemma to bound the tail of $A$.
\begin{lemma}\label{lemma:lambda}
    Let $N \in \mathbb{N}^2$ be such that $N \geq N^{\mathrm{Jac}}$. Let $\lambda_n$ be as in \eqref{eq:deflambda}. Then
\begin{align*}
        \frac{1}{\lambda_n} \leq 
	\frac{1}{\lambda_{\mathrm{min}} (N)}  
	\qquad\text{for all }n\notin I^+_{N}.
\end{align*}
\end{lemma}
\begin{proof}
   From Remark \ref{rem:lambda}, we know that $\lambda_n$ is positive and increasing in both $n_1$ and $n_2$ for $n\notin I^+_N$, since $I^+_{N^{\mathrm{Jac}}}\subset I^+_N$. Hence, the smallest values of $\lambda_n$ for $n\notin I^+_N$ is attained either at $n=(N_1+1,0)$ or at $n_2=N_2+1$ and some $0\leq n_1\leq \lceil \frac{c}{q_1}\rceil$. This proves the assertion on the inverse $\lambda_n^{-1}$.
\end{proof}
We now bound the second summation in~\eqref{eq:AF} as follows:
\begin{align*}
    \sum_{ n \notin I^+_{N^\mathrm{Alias}}}\left|\left(AF(\bar{a})\right)_n\right| \omega_n&
=\sum_{n\notin I_{N^\mathrm{Alias}}^+ } \frac{\left|\bar{b}_{n}\right|}{\lambda_n} \multi_n \nu^{n} \\
& \leq \frac{\widehat{C}_{\bar{\rho}}}{\lambda_\mathrm{min}(N^\mathrm{Alias})}  \sum_{n\notin I_{N^\mathrm{Alias}}^+ }\multi_n\frac{\nu^{n}}{\bar{\nu}^{n}} \qquad \text{ (Lemmas \ref{lemma:boundb} and \ref{lemma:lambda})}\\
& = \frac{\widehat{C}_{\bar{\rho}}}{\lambda_\mathrm{min}(N^\mathrm{Alias})} f_\mathrm{geo}\left(\frac{\nu}{\bar{\nu}},N^\mathrm{Alias}\right) \qquad \text{(Lemma \ref{lemma:geometric})}.
\end{align*}
Both the bound above and the formula for $Y^{N^{\mathrm{Alias}}}$ are evaluated using interval arithmetic.
We take the upper bound of the interval containing the sum of these to determine $Y$:
$$
Y :=\text{upperbound}\left(Y^{N^{\mathrm{Alias}}}+\frac{\widehat{C}_{\bar{\rho}}}{\lambda_\mathrm{min}(N^\mathrm{Alias})}  f_\mathrm{geo}\left(\frac{\nu}{\bar{\nu}},N^\mathrm{Alias}\right)\right).
$$
In the next subsections we will also use interval arithmetic to compute the $Z$-bound and the $W$-bound. However, in the derivation for these bounds we will refrain from repeating 
that we are dealing with intervals and that we take the upper bound of the intervals in the end.

\subsection{Computing $Z$}
\label{sec:Zbound}

Since the norm on $X$ is a weighted $\ell^1$-norm, we can express the operator norm as (e.g.\ Lemma 2.13 in \cite{Breden_gPC})
\begin{align*}
    \|I-ADF(\bar{a})\|_{B(X)}=\sup_{k \in \mathbb{N}^2}\frac{\||[I-ADF(\bar{a})]e_k\|_{\nu}}{\|e_k\|_{\nu}} ,
\end{align*}
where $(e_k)_n=\delta_{kn}$ for $k,n \in \mathbb{N}^2$ and where $\delta$ denotes the Kronecker delta. Hence, we consider the norm of the ``columns'' of the operator $I-ADF(\bar{a})$ imagined as an infinite matrix in order to determine the $Z$-bound. For that purpose, we introduce some notation. Let $\mathbbm{1}_{S}$ be the indicator function that is $1$ if $S$ is true and $0$ otherwise. 
We have for $k,n \in \mathbb{N}^2$
\begin{align}\label{eq:IADF}
    \left(ADF(\bar{a})e_k\right)_n
    &=\sum_{j\in I_{N^\mathrm{Jac}}^+} A^{N^\mathrm{Jac}}_{nj}\lambda_j\delta_{kj}\mathbbm{1}_{\{n \in I_{N^\mathrm{Jac}}^+\}}+\delta_{kn}\mathbbm{1}_{\{n \notin I_{N^\mathrm{Jac}}^+\}}
	 +(ADG(\bar{a}) e_k))_n.
\end{align}
To make the expression for $DG(\bar{a}) e_k$ more explicit, we introduce some notation related to the fact that we are working with cosine series:
\begin{equation*}
\begin{aligned}[c]
h_1(k)&:=(k_1,k_2), \\
h_2(k)&:=(-k_1,k_2),
\end{aligned}
\qquad\qquad\qquad
\begin{aligned}[c]
 h_3(k)&:=(k_1,-k_2),\\
h_4(k)&:=(-k_1,-k_2).\\
\end{aligned}
\end{equation*}
\begin{remark}\label{rem:rewritesum}
The definition of $h_i$ for $i=1,2,3,4$ is chosen such that, with $\multi_k$ defined in  \eqref{eq:defgn}, for any $a \in \mathbb{C}^{\mathbb{Z}^2}$ such that $\sum_{n\in\mathbb{Z}^2} |a_n|<\infty$ it holds that
    \begin{align*}
        \sum_{n\in \mathbb{Z}^2}a_n=\sum_{n\in\mathbb{N}^2}\frac{\multi_n}{4}\sum_{i=1}^4a_{h_i(n)} .
    \end{align*}
\end{remark}
For the derivative of the nonlinearity we have $DG(a)\tilde{a}=G'(a)\ast  \tilde{a}$. Denoting $\bar{b}'=G'(\bar{a})$, we thus infer that $DG(\bar{a}) e_k = \bar{b}' \ast e_k$, hence (using Remark \ref{rem:rewritesum}) some bookkeeping leads to 
\begin{align*}
    (ADG(\bar{a})e_k)_n &=\sum_{j \in I_{N^\mathrm{Jac}}^+}A_{nj}^{N^\mathrm{Jac}}\sum_{m\in \mathbb{Z}^2}\bar{b}'_{|j-m|}(e_k)_{|m|}\mathbbm{1}_{\{n \in I_{N^\mathrm{Jac}}^+\}}+\frac{1}{\lambda_n}\sum_{m\in \mathbb{Z}^2}\bar{b}'_{|n-m|}(e_k)_{|m|}\mathbbm{1}_{\{n \notin I_{N^\mathrm{Jac}}^+\}} \\
    &=\sum_{j \in I_{N^\mathrm{Jac}}^+}A_{nj}^{N^\mathrm{Jac}}\frac{\multi_k}{4}\sum_{i=1}^4 \bar{b}'_{|j - h_i(k)|}\mathbbm{1}_{\{n \in I_{N^\mathrm{Jac}}^+\}}+\frac{1}{\lambda_n}\frac{\multi_k}{4}\sum_{i=1}^4 \bar{b}'_{|n - h_i(k)|}\mathbbm{1}_{\{n \notin I_{N^\mathrm{Jac}}^+\}}.
\end{align*}
Substituting this into \eqref{eq:IADF}, we obtain
\begin{align*}
    \left([I-ADF(\bar{a})]e_k\right)_n=&\Biggl((e_k)_n-\sum_{j\in I_{N^\mathrm{Jac}}^+} A^{N^\mathrm{Jac}}_{nj}\lambda_j\delta_{kj}-\sum_{j \in I_{N^\mathrm{Jac}}^+}A_{nj}^{N^\mathrm{Jac}}\frac{\multi_k}{4}\sum_{i=1}^4 \bar{b}'_{|j - h_i(k)|}\Biggr)\mathbbm{1}_{\{n \in I_{N^\mathrm{Jac}}^+\}}\\
    &\hspace{2cm} +\left((e_k)_n-\delta_{kn} -\frac{1}{\lambda_n}\frac{\multi_k}{4}\sum_{i=1}^4 \bar{b}'_{|n - h_i(k)|}\right)\mathbbm{1}_{\{n \notin I_{N^\mathrm{Jac}}^+\}}.
\end{align*}
Since $(e_k)_n=\delta_{kn}$ for $k,n \in \mathbb{N}$
and $\|e_k\|_\nu=\omega_k$, we infer that
\begin{align}\label{eq:IADFsums}
    \frac{\||[I-ADF(\bar{a})]e_k\|_{\nu}}{\|e_k\|_{\nu}}=&\sum_{n \in I_{N^\mathrm{Jac}}^+}\Biggl|\delta_{kn}-\sum_{j\in I_{N^\mathrm{Jac}}^+} A^{N^\mathrm{Jac}}_{nj}\lambda_j\delta_{kj}-\sum_{j \in I_{N^\mathrm{Jac}}^+}A_{nj}^{N^\mathrm{Jac}}\frac{\multi_k}{4}\sum_{i=1}^4 \bar{b}'_{|j - h_i(k)|}\Biggr|\frac{\omega_n}{\omega_k} \nonumber\\
    &\hspace{2cm} +\sum_{n \notin I_{N^\mathrm{Jac}}^+}\frac{1}{\lambda_n}\left|\frac{\multi_k}{4}\sum_{i=1}^4 \bar{b}'_{|n - h_i(k)|}\right|\frac{\omega_n}{\omega_k}.
\end{align}
We estimate both summations separately, starting with the first one.

\subsubsection{$n\in I_{N^\mathrm{Jac}}^+$}\label{sec:Zboundsmalln}

In the first summation of \eqref{eq:IADFsums}, the first two terms are only non-zero for $k \in I^+_{N^\mathrm{Jac}}$. Therefore, we rewrite the summation once more:
\begin{align}\label{eq:smallnZ}
& \sum_{n \in I_{N^\mathrm{Jac}}^+}\Biggl|\delta_{kn}-\sum_{j\in I_{N^\mathrm{Jac}}^+} A^{N^\mathrm{Jac}}_{nj}\lambda_j\delta_{kj}-\sum_{j \in I_{N^\mathrm{Jac}}^+}A_{nj}^{N^\mathrm{Jac}}\frac{\multi_k}{4}\sum_{i=1}^4 \bar{b}'_{|j - h_i(k)|}\Biggr|\frac{\omega_n}{\omega_k} \nonumber\\
    &\hspace{2cm} =\sum_{n\in I_{N^\mathrm{Jac}}^+}\Biggl|[\delta_{kn}-A^{N^\mathrm{Jac}}_{nk}\lambda_k]\mathbbm{1}_{\{k\in I_{N^\mathrm{Jac}}^+\}}-\sum_{j\in I_{N^\mathrm{Jac}}^+}A^{N^\mathrm{Jac}}_{nj}\frac{\multi_k}{4}\sum_{i=1}^4 \bar{b}'_{|j - h_i(k)|}\Biggr|\frac{\omega_n}{\omega_k}.
\end{align}
The first part, $[\delta_{kn}-A^{N^\mathrm{Jac}}_{nk}\lambda_k]\mathbbm{1}_{\{k\in  I_{N^\mathrm{Jac}}^+\}}$, can be computed explicitly for all $n\in I^+_{N^\mathrm{Jac}}$.
The second part, $\sum_{j\in I_{N^\mathrm{Jac}}^+}A^{N^\mathrm{Jac}}_{nj}\frac{\multi_k}{4}\sum_{i=1}^4 \bar{b}'_{|j - h_i(k)|}$, is computed using \eqref{eq:bnprimecompute} for ``small'' $k\in I^+_{N^\mathrm{Col}}$, whereas it is estimated for all ``large'' $k \notin I^+_{N^\mathrm{Col}}$. Here $N^\mathrm{Col} \geq N^\mathrm{Jac}$ indicates where we switch from essentially evaluating ``columns'' of the linear operator to uniformly estimating. 
 
After computing the finite sum in~\eqref{eq:smallnZ} for all $k \in I^+_{N^\mathrm{Col}}$, we are left with uniformly bounding the second summation in the right-hand side of~\eqref{eq:smallnZ} for 
$k \notin I^+_{N^\mathrm{Col}}$.  Before we proceed we analyze
\[
  \mumu(j,k) := \frac{1}{4} \sum_{i=1}^4  	\frac{1}{\bar{\nu}^{|j-h_i(k)|}\nu^k}
\]
for $j \in I^+_{N^\mathrm{Jac}}$ and $k \notin I^+_{N^\mathrm{Col}}$. 
For any $N \in \mathbb{N}^2$ we introduce the ``boundary'' $\muI_N \subset \mathbb{N}^2$ of $I^+_{N}$:
\[
   \muI_N :=  
   \bigl\{(k_1,N_2+1) :  0 \leq k_1 \leq N_1 \bigr\} 
   \,\cup\,
   \bigl\{(N_1+1,k_2) :  0 \leq k_2 \leq N_2 \bigr\} .
\]
We then define
\[
  \muhat(j,N) := \max_{k\in \muI_N} \mumu(j,k).
\] 

\begin{lemma}\label{lemma:nuest}
    Let $N \in \mathbb{N}^2$ and $j \in I^+_{N}$. Then 
	$ \mumu(j,k)  \leq \muhat(j,N)$ for all $k \notin I^+_{N}$,
\end{lemma}
\begin{proof}
	Fix $j \in I^+_{N}$.
	We first observe that 
	$\mumu(j,k)= \frac{1}{4} \sum_{i=1}^4  	\frac{1}{\bar{\nu}^{|h_i(j)-k|}\nu^k}$. Since $\bar{\nu} > \nu 
	\geq 1$, each of the four summands decreases in $k_1$ for $k_1 \geq N_1$ and 
	decreases in $k_2$ for $k_2 \geq N_2$.
	Hence the supremum of $\mumu(j,k)$ over $k \notin I^+_{N}$
	is attained for some $k \in \muI_N$. This proves the assertion.		
\end{proof}

We are now ready to estimate the second summation in the right-hand side of \eqref{eq:smallnZ}. By using \eqref{eq:bnprimecompute} and Lemma~\ref{lemma:nuest}, we conclude that for any $k \notin I^+_{N^\mathrm{Col}}$
\begin{align*}
  \sum_{n\in I_{N^\mathrm{Jac}}^+}\sum_{j\in I_{N^\mathrm{Jac}}^+}\left|A_{nj}^{N^\mathrm{Jac}}\right|\left|\frac{\multi_k}{4}\sum_{i=1}^4 \bar{b}'_{|j - h_i(k)|}\right|\frac{\omega_n}{\omega_k} 
  \hspace*{-2cm} \\
  &\leq   \sum_{n\in I_{N^\mathrm{Jac}}^+}\sum_{j\in I_{N^\mathrm{Jac}}^+}\left|A_{nj}^{N^\mathrm{Jac}}\right|\frac{\multi_k}{4}\sum_{i=1}^4\frac{\widehat{C}'_{\bar{\rho}}}{\bar{\nu}^{|j-h_i(k)|}}\frac{\omega_n}{\multi_k\nu^k}\\
    &\leq 
  \widehat{C}'_{\bar{\rho}}\sum_{n\in I_{N^\mathrm{Jac}}^+}\sum_{j\in I_{N^\mathrm{Jac}}^+}\left|A_{nj}^{N^\mathrm{Jac}}\right|  \muhat(j,N^\mathrm{Col}) \, \omega_n,
\end{align*}
where we have used that $N^\mathrm{Col} \geq N^\mathrm{Jac}$.
We note that the right-hand side does not depend on~$k$ and requires only a finite computation to evaluate.

\subsubsection{$n \notin I_{N^\mathrm{Jac}}^+$}\label{sec:Zboundlargen}
We now investigate the second summation in the right-hand side of \eqref{eq:IADFsums}, which is given by
\begin{align}\label{eq:largenY}
\sum_{n \notin I_{N^\mathrm{Jac}}^+ }\frac{1}{\lambda_n}\left|\frac{\multi_k}{4}\sum_{i=1}^4 \bar{b}'_{|n - h_i(k)|}\right|\frac{\omega_n}{\omega_k}.
\end{align}
We start by considering ``small'' $k$, i.e., $k \in I^+_{N^\mathrm{Col}}$. 
For these values of $k$, we split the summation in~\eqref{eq:largenY} into ``small'' values of $n \in I^+_{N^\mathrm{Row}}$, where we simply evaluate the term, and ``large'' values of $n \notin I^+_{N^\mathrm{Row}}$, where we need an estimate.
Here $N^\mathrm{Row} \geq N^\mathrm{Jac}$ indicates where we switch from explicitly calculating ``rows'' to estimating uniformly
using the rougher enclosure of~$\bar{b}_n'$. 
Hence we write 
\begin{align*}
    \sum_{n \notin I_{N^\mathrm{Jac}}^+}\frac{1}{\lambda_n}\left|\frac{\multi_k}{4}\sum_{i=1}^4 \bar{b}'_{|n - h_i(k)|}\right|\frac{\omega_n}{\omega_k}=&\sum_{\substack{n\in I_{N^\mathrm{Row}}^+\\n \notin I_{N^\mathrm{Jac}}^+}} \frac{1}{\lambda_n}\left|\frac{\multi_k}{4}\sum_{i=1}^4 \bar{b}'_{|n - h_i(k)|}\right|\frac{\omega_n}{\omega_k} \nonumber\\[-4mm]
    & \hspace{3cm}
	+\sum_{n\notin I^+_{N^\mathrm{Row}} }\frac{1}{\lambda_n}\left|\frac{\multi_k}{4}\sum_{i=1}^4 \bar{b}'_{|n - h_i(k)|}\right|\frac{\omega_n}{\omega_k}.
\end{align*}
The first term in the right-hand side is calculated (for each $k \in I^+_{N^\mathrm{Col}}$). 
For the second term we derive an estimate.
Indeed, we observe that 
\[
|n-h_i(k)|\geq \bigl||n|-|h_i(k)|\bigr|=\bigl||n|-|k|\bigr|\geq |n|-|k|=n-k.
\]
We then apply Lemmas~\ref{lemma:boundb} and~\ref{lemma:lambda}, and subsequently Lemma~\ref{lemma:geometric} to obtain
\begin{align*}
   \sum_{n\notin I^+_{N^\mathrm{Row}}}\frac{1}{\lambda_n}\left|\frac{\multi_k}{4}\sum_{i=1}^4 \bar{b}'_{|n - h_i(k)|}\right|\frac{\omega_n}{\omega_k}&\leq \frac{1}{\lambda_\mathrm{min}(N^\textrm{Row})}\sum_{n\notin I^+_{N^\mathrm{Row}}}\frac{1}{4}\sum_{i=1}^4\frac{\widehat{C}'_{\bar{\rho}}}{\bar{\nu}^{|n-h_i(k)|}}\frac{\multi_n\nu^n}{\nu^k}\\
      &\leq \frac{\widehat{C}'_{\bar{\rho}}}{\lambda_\mathrm{min}(N^\textrm{Row})}\sum_{n\notin I^+_{N^\mathrm{Row}}}\frac{\multi_n}{4}\sum_{i=1}^4\frac{1}{\bar{\nu}^{n-k}}\frac{\nu^n}{\nu^k} 
	\\&=
	\frac{\widehat{C}'_{\bar{\rho}}\bar{\nu}^k}{\lambda_\mathrm{min}(N^\textrm{Row})\nu^k}f_\mathrm{geo}\left(\frac{\nu}{\bar{\nu}},N^\mathrm{Row}\right).
     \end{align*}
We note that this estimate holds for any $k\in \mathbb{N}^2$, but we will use it only for $k \in I^+_{N^\mathrm{Col}}$. 

We are left with establishing a uniform bound on~\eqref{eq:largenY} for $k \notin I^+_{N^\mathrm{Col}}$.
In order to take advantage of computational power to tighten the bound, we introduce another computational parameter
$N^\mathrm{Tail}$ with $N^\mathrm{Jac}\leq N^\mathrm{Tail}\leq N^\mathrm{Col}$
and split
\begin{align}
\sum_{n \notin I_{N^\mathrm{Jac}}^+ }\frac{1}{\lambda_n}\left|\frac{\multi_k}{4}\sum_{i=1}^4 \bar{b}'_{|n - h_i(k)|}\right|\frac{\omega_n}{\omega_k}
&=
\sum_{\substack{n\in I_{N^\mathrm{Tail}}^+\\n \notin I_{N^\mathrm{Jac}}^+}}
\frac{1}{\lambda_n}\left|\frac{\multi_k}{4}\sum_{i=1}^4 \bar{b}'_{|n - h_i(k)|}\right|\frac{\omega_n}{\omega_k}
\nonumber\\[-5mm]
&\hspace{2.5cm}+
\sum_{n \notin I_{N^\mathrm{Tail}}^+ }\frac{1}{\lambda_n}\left|\frac{\multi_k}{4}\sum_{i=1}^4 \bar{b}'_{|n - h_i(k)|}\right|\frac{\omega_n}{\omega_k}.
\label{eq:splitNtail}
\end{align}
For the first term in the right-hand side of~\eqref{eq:splitNtail}
we proceed as in Section~\ref{sec:Zboundsmalln} and use Lemma~\ref{lemma:nuest} 
to obtain
\begin{align*}
\sum_{\substack{n\in I_{N^\mathrm{Tail}}^+\\n \notin I_{N^\mathrm{Jac}}^+}}
\frac{1}{\lambda_n}\left|\frac{\multi_k}{4}\sum_{i=1}^4 \bar{b}'_{|n - h_i(k)|}\right|\frac{\omega_n}{\omega_k}
&\leq
\widehat{C}'_{\bar{\rho}}
\sum_{\substack{n\in I_{N^\mathrm{Tail}}^+\\n \notin I_{N^\mathrm{Jac}}^+}}
\frac{1}{\lambda_n}\left|\frac{1}{4}\sum_{i=1}^4 \frac{1}{\bar{\nu}^{|n-h_i(k)|}\nu^k} \right|\omega_n  
\\
&\leq
\widehat{C}'_{\bar{\rho}}
\sum_{\substack{n\in I_{N^\mathrm{Tail}}^+\\n \notin I_{N^\mathrm{Jac}}^+}}
\frac{\muhat(n,N^\mathrm{Col}) \, \omega_n}{\lambda_n} .
\end{align*}
For the second term in the right-hand side of~\eqref{eq:splitNtail}
we use Lemma~\ref{lemma:lambda} to infer that
\begin{align*}
     \sum_{n\notin I^+_{N^\mathrm{Tail}}}\frac{1}{\lambda_n}\left|\frac{\multi_k}{4}\sum_{i=1}^4 \bar{b}'_{|n - h_i(k)|}\right|\frac{\omega_n}{\omega_k}&\leq \frac{1}{\lambda_\textrm{min}(N^\textrm{Tail})}\sum_{n\in \mathbb{N}^2}\frac{1}{4}\sum_{i=1}^4\left|\bar{b}'_{|n-h_i(k)|}\right|\frac{\multi_n\nu^n}{\nu^k} \\
          &\leq \frac{1}{\lambda_\textrm{min}(N^\textrm{Tail})}\sum_{n\in \mathbb{N}^2}\frac{1}{4}\sum_{i=1}^4\left|\bar{b}'_{|h_i(n)-k|}\right|\multi_n\nu^{|n-k|} \, .
\end{align*}
Next we rewrite the sum over 
$n\in\mathbb{N}^2$ into a sum over  $n\in \mathbb{Z}^2$ (see Remark \ref{rem:rewritesum}) and use  \eqref{eq:bnprimecompute} and Lemma~\ref{lemma:geometric} to obtain the explicit uniform estimate
\begin{align*}
\sum_{n\in \mathbb{N}^2}\frac{1}{4}\sum_{i=1}^4\left|\bar{b}'_{|h_i(n)-k|}\right|\multi_n\nu^{|n-k|} 
     &= \sum_{n\in \mathbb{Z}^2}\left|\bar{b}'_{|n-k|}\right|\nu^{|n-k|}\\
     &= \|\bar{b}'\|_{\nu}\\
&= \sum_{n\in I_{N^\mathrm{Alias}}} |\bar{b}'_{|n|}|\nu^{|n|} +\sum_{n\notin I_{N^\textrm{Alias}}}|\bar{b}'_{|n|}|\nu^{|n|}\\
     &= \sum_{n\in I^+_{N^\mathrm{Alias}}} |\bar{b}'_{n}|\omega_n+\sum_{n\notin I_{N^\textrm{Alias}}}\frac{\widehat{C}'_{\bar{\rho}}}{\bar{\nu}^{|n|}}\nu^{|n|}\\
     &\leq  \sum_{n\in I^+_{N^\mathrm{Alias}}} |\bar{b}'_{n}|\omega_n+\widehat{C}'_{\bar{\rho}}f_\mathrm{geo}\left(\frac{\nu}{\bar{\nu}},N^\mathrm{Alias}\right).
\end{align*}
We note that this estimate again holds for all $k \in \mathbb{N}^2$, but it is used only for $k\notin I^+_{N^\mathrm{Col}}$.

\subsubsection{Overall $Z$-bound}
To summarize, to estimate the operator norm of $I-ADF(\bar{a})$
the strategy in this section has been to compute where we can and to estimate (uniformly) where we have to. 
We collect the results from Sections~\ref{sec:Zboundsmalln} and~\ref{sec:Zboundlargen}
for $k \in I^+_{N^\mathrm{Col}}$ into
\begin{align*}
    Z^\mathrm{Col}&:=\underset{k \in I^+_{N^\mathrm{Col}}}{\max}\Bigg\{\sum_{n\in I_{N^\mathrm{Jac}}^+}\Biggl|[\delta_{kn}-A^{N^\mathrm{Jac}}_{nk}\lambda_k]\mathbbm{1}_{\{k\in I_{N^\mathrm{Jac}}^+\}}-\sum_{j\in I_{N^\mathrm{Jac}}^+}A^{N^\mathrm{Jac}}_{nj}\frac{\multi_k}{4}\sum_{i=1}^4 \bar{b}'_{|j - h_i(k)|}\Biggr|\frac{\omega_n}{\omega_k}\\
    &\hspace{2.5cm}
	+\sum_{\substack{n\in I^+_{N^\mathrm{Row}}\\n \notin I^+_{N^\mathrm{Jac}}}} \frac{1}{\lambda_n}\left|\frac{\multi_k}{4}\sum_{i=1}^4 \bar{b}'_{|n - h_i(k)|}\right|\frac{\omega_n}{\omega_k} 
	\:+\:
	\frac{\widehat{C}'_{\bar{\rho}}\bar{\nu}^k}{\lambda_\mathrm{min}(N^\textrm{Row})\nu^k}f_\mathrm{geo}\left(\frac{\nu}{\bar{\nu}},N^\mathrm{Row}\right)\Bigg\},
\end{align*}
and those for $k \notin I^+_{N^\mathrm{Col}}$ into
\begin{align*}
    Z^\mathrm{Est}&:=\widehat{C}'_{\bar{\rho}}\sum_{n\in I_{N^\mathrm{Jac}}^+}\sum_{j\in I_{N^\mathrm{Jac}}^+}\left|A_{nj}^{N^\mathrm{Jac}}\right|
    \muhat(j,N^\mathrm{Col}) \,\omega_n 
	+ \: \widehat{C}'_{\bar{\rho}}
	\sum_{\substack{n\in I_{N^\mathrm{Tail}}^+\\n \notin I_{N^\mathrm{Jac}}^+}}
	\frac{\muhat(n,N^\mathrm{Col}) \, \omega_n}{\lambda_n} 
	\\
    &\hspace{5cm}
	+\: \frac{1}{\lambda_\textrm{min}(N^\textrm{Tail})}\Biggl[ \, \sum_{n\in I^+_{N^\mathrm{Alias}}} |\bar{b}'_{n}|\omega_n+\widehat{C}'_{\bar{\rho}}f_\mathrm{geo}\left(\frac{\nu}{\bar{\nu}},N^\mathrm{Alias}\right)\Biggr] .
\end{align*}
In both expressions above we evaluate $\bar{b}'_n$ in an interval arithmetic sense through~\eqref{eq:bnprimecompute}.
Finally, we set 
 $  Z:=\max\{Z^\mathrm{Col},Z^\mathrm{Est}\}$.

\begin{remark}\label{rem:5blocks}
Observe that $Z^\mathrm{Col}$ consists of three terms and $Z^\mathrm{Est}$ consists of two terms, i.e., 
$Z^\mathrm{Col}:=\max\{ Z^\mathrm{Col}_{1k}+Z^\mathrm{Col}_{2k}+Z^\mathrm{Col}_{3k} : k \leq N^{Col} \}$ and $Z^\mathrm{Est}:=Z^\mathrm{Est}_1+Z^\mathrm{Est}_2+Z^\mathrm{Est}_3$. Essentially we have divided the estimate for the operator norm $\|I-ADF(\bar{a})\|_{B(X)}$ into analyzing the six blocks of $I-ADF(\bar{a})$ outlined schematically in Figure \ref{fig:tikz}.

We see that values $\bar{b}'_n$ appear in $Z_1^\mathrm{Col}$, $Z_2^\mathrm{Col}$ and $Z_3^\mathrm{Est}$. 
This makes sense only if for most $n$ appearing in those summations,
the values $\bar{b}'_n$ are computed via the discrete Fourier transform with tight aliasing error as described in Section \ref{sec:alias} rather than using the rough uniform estimate from Section \ref{sec:generalbound}.
This is trivially satisfied for the summation in $Z_3^\mathrm{Est}$, but for $Z_1^\mathrm{Col}$ and $Z_2^\mathrm{Col}$ this implies a (mild) relation between $N^\mathrm{Alias}$ on the one hand and $N^\mathrm{Col}$ and $N^\mathrm{Row}$ on the other.  
This is indicated schematically by the diagonal strip of width $N^\mathrm{Alias}$, representing the Toeplitz matrix of $\Pi^{N^\mathrm{Alias}} \bar{b}'$, which is part of $DF(\bar{a})$. The diagonal strip should cover a ``significant part'' of the blocks corresponding to $Z_1^\mathrm{Col}$ and $Z_2^\mathrm{Col}$. Similarly, it is a waste of computational resources to choose $N^\mathrm{Alias}$ too large and hence compute accurately values of $\bar{b}'_n$ not needed to determine $Z_1^\mathrm{Col}$ and $Z_2^\mathrm{Col}$.
\end{remark}

\begin{figure}[t]
\centering
\includegraphics[width=0.55\textwidth]{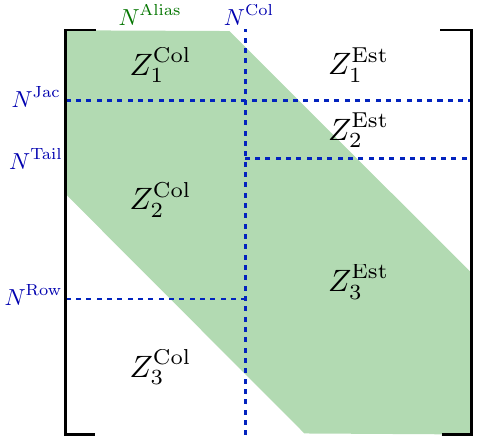}
\caption{Schematic representation of the different blocks of $I-ADF(\bar{a})$ appearing in the analysis of ${\|I-ADF(\bar{a})\|_{B(X)}}$ and their relation with $N^\mathrm{Jac}$, $N^\mathrm{Col}$, $N^\mathrm{Row}$ and $N^\mathrm{Tail}$. The diagonal band represents a Toeplitz matrix of width $N^\mathrm{Alias}$, see Remark~\ref{rem:5blocks}.}
\label{fig:tikz}
\end{figure}

\subsection{Computing $W$}
\label{sec:Wbound}

Fixing $r^\ast>0$, we need a $W\in \mathbb{R}^+$ such that
$$
\|A[DF(\bar{a}+w)-DF(\bar{a})]\|_{B(X)}\leq W\|w\|_\nu \qquad \text{whenever } \|w\|_\nu \leq r^\ast.
$$
For any $v$ such that $\|v\|_\nu\leq 1$, we define $g_v(w):=ADF(\bar{a}+w)v$, so that 
\begin{align*}
    g_v(w)-g_v(0)=A[DF(\bar{a}+w)-DF(\bar{a})]v,
\end{align*}
and $\|A[DF(\bar{a}+w)-DF(\bar{a})]\|_{B(X)}=\sup_{\|v\|_\nu \leq 1}  \|g_v(w)-g_v(0) \|_\nu$.
By the mean value theorem for Banach spaces (cf. Theorem 1.1.1 in \cite{Hormander}),
we have that
$\|g_v(w)-g_v(0)\|_{\nu}\leq M_v \|w\|_\nu$ whenever $\|w\|_\nu\leq r^\ast$,
where $M_v:=\sup_{\|\tilde{w}\|_\nu\leq r^\ast} \|Dg_v(\tilde{w})\|_{B(X)}$.
Using the definition of the operator norm, we rewrite $M_v$ as 
\begin{align*}
    M_v&
	= \underset{\substack{\|\tilde{v}\|_\nu\leq 1\\ \|\tilde{w}\|_\nu \leq r^\ast}}{\sup}\|Dg_v(\tilde{w})\tilde{v}\|_{\nu}
	=\underset{\substack{\|\tilde{v}\|_\nu\leq 1\\ \|\tilde{w}\|_\nu \leq r^\ast}}{\sup}\|AD^2F(\bar{a}+\tilde{w})[\tilde{v},v]\|_{\nu}
	=\underset{\substack{\|\tilde{v}\|_\nu\leq 1\\ \|\tilde{w}\|_\nu \leq r^\ast}}{\sup}\|Ae^{\bar{a}+\tilde{w}}\ast \tilde{v}\ast v\|_{\nu}.
\end{align*}
We use the Banach algebra property as given in \eqref{eq:banachalgebra} to obtain, for any $v$ with $\|v\|_\nu\leq 1$,
\begin{align*}
    \|g_v(w)-g_v(0)\|_{\nu}&\leq 
		\|A\|_{B(X)} \underset{\substack{\|\tilde{v}\|_\nu\leq 1\\ \|\tilde{w}\|_\nu\leq
	 r^\ast}}{\sup}
\|e^{\bar{a}+\tilde{w}}\|_\nu \, \|\tilde{v}\|_\nu \, \|v\|_\nu \, \|w\|_\nu \leq \|A\|_{B\left(X\right)} \, e^{r^{*}}\left\|e^{\bar{a}}\right\|_\nu\|w\|_\nu .
\end{align*}
Here we have used that $e^{\bar{a}+\tilde{w}}=e^{\bar{a}} \ast e^{\tilde{w}}$
and $\|e^{\tilde{w}}\|_\nu \leq e^{\|\tilde{w}\|_\nu}$.
We note that although $A$ is described as an operator from $\hat{X}$ to $X$ in Theorem~\ref{thm:radii}, we are free to interpret it as a bounded operator from $X$ to $X$ here.
To determine $\|A\|_{B\left(X\right)} $, we compute the matrix norm 
\[
  \|A^{N^{\mathrm{Jac}}}\|_{B(X)} 
  = 
  \max_{k \in I^+_{N^{\mathrm{Jac}}}}\frac{\|A^{N^{\mathrm{Jac}}}e_k\|_{\nu}}{\|e_k\|_{\nu}},
\]
and then use that (see Lemma~\ref{lemma:lambda})
\begin{align}\label{eq:Acomp}    \|A\|_{B\left(X\right)}=\max\left\{\|A^{N^\textup{Jac}}\|_{B(X)},\lambda^{-1}_\mathrm{min}(N^\mathrm{Jac})\right\}.
\end{align}
To estimate $\bar{b}''=e^{\bar{a}}$, we
proceed as before to bound the Fourier coefficients of $\mathcal{G}'' \circ \bar{u}$, see 
Remark~\ref{remark:barbprime}.
By repeating part of the construction described in Section~\ref{sec:Ybound}, 
we obtain
\begin{align*}
    \|e^{\bar{a}}\|_\nu\leq \sum_{n\in I^+_{N^\mathrm{Alias}}}\left(|\bar{b}_n^{''\mathrm{FFT}}|+\widehat{C}''_{\bar{\rho}} \bar{\varepsilon}_n \right)\omega_n+\widehat{C}''_{\bar{\rho}}f_\mathrm{geo}\left(\frac{\nu}{\bar{\nu}},N^\mathrm{Alias}\right).
\end{align*}
Setting \begin{align*}
    \widehat{W}:=\|A\|_{B\left(X\right)} \, \left(\sum_{n\in I^+_{N^\mathrm{Alias}}}\left(|\bar{b}_n^{''\mathrm{FFT}}|+\widehat{C}''_{\bar{\rho}} \bar{\varepsilon}_n\right)\omega_n+\widehat{C}''_{\bar{\rho}}f_\mathrm{geo}\left(\frac{\nu}{\bar{\nu}},N^\mathrm{Alias}\right)\right),
    \end{align*}
we conclude that we may choose, for any fixed $r^\ast>0$,
\begin{align}\label{eq:What}
W = \widehat{W}e^{r^{*}}.
\end{align}
Finally, we remark that we have used the estimate $\| e^{\bar{a}+\tilde{w}} \|_{\nu} \leq 
\|e^{\bar{a}}\|_\nu e^{r^\ast}$ and that this convenient property of our exponential nonlinearity leads to substantially better results than the  more naive bound 
$\| e^{\bar{a}+\tilde{w}} \|_{\nu} \leq 
 e^{\|\bar{a}\|_\nu + r^\ast}$.


\section{Results}
\label{sec:results}
In this section, we prove existence of different solutions to \eqref{eq:SBeq} with boundary conditions \eqref{eq:SBBC}. As explained in the introduction, these correspond to periodic waves on an infinite strip, traveling at speed~$c$.
The codes that are used to obtain the results can be found in \cite{github}. 
All computations were carried out in MATLAB (2024a) on an Apple M1 Pro CPU with 32 GB RAM. In order to make the calculations rigorous, the interval arithmetic library Intlab~\cite{IntLab_Rump} is used.  We report floating point numbers displaying four decimals, with the precise values available in the code. At the end of this section we discuss how we found numerical approximations of solutions. We start by proving the theorem that was stated in the introduction. 
\begin{proof}[Proof of Theorem \ref{thm:examplesol}]
We choose $\nu=(1+10^{-7},1+10^{-7})$
 and $\bar{\rho}=(0.09531,0.09531)$. The reasons for this choice and the selection of the other values of the parameters in this proof are discussed in Remark \ref{rem:choiceparam}. The values of the different truncation parameters $N$ are presented in Table \ref{tab:examplesol}. 
\begin{table}
\centerline{\renewcommand{\arraystretch}{1.2}
\begin{tabular}{lllllll}
\hline
 \multicolumn{1}{|l|}{$N^\mathrm{Gal}$} & \multicolumn{1}{l|}{$N^\mathrm{Jac}$}& 
 \multicolumn{1}{l|}{$N^\mathrm{Alias}$}& \multicolumn{1}{l|}{$N^\mathrm{FFT}$}& \multicolumn{1}{l|}{$N^\mathrm{Col}$}& \multicolumn{1}{l|}{$N^\mathrm{Row}$} &
 \multicolumn{1}{l|}{$N^\mathrm{Tail}$}\\ \hline
\multicolumn{1}{|l|}{$(130,130)$} & \multicolumn{1}{l|}{$(65,65)$}& \multicolumn{1}{l|}{$(400,400)$}& \multicolumn{1}{l|}{$(1024,1024)$}& \multicolumn{1}{l|}{$(300,300)$}& \multicolumn{1}{l|}{$(800,800)$} &\multicolumn{1}{l|}{$(140,140)$}\\ \hline
\end{tabular}
}
\caption{Overview of the different values of $N$ used for proving existence of solutions to \eqref{eq:SBeq} with boundary conditions \eqref{eq:SBBC} in Theorem \ref{thm:examplesol}.}
\label{tab:examplesol}
\end{table}
The Fourier coefficients $\bar{a}$ of the approximate solution $\bar{u}$ (depicted in Figure~\ref{fig:examplesol}) can be found in the code. 
The values for the bounds are (for $r^\ast=6.3605\cdot 10^{-3}$; see Remark~\ref{rem:rstar})
\begin{align*}
    Y=2.4708 \cdot 10^{-8}, \quad Z=2.9545 \cdot 10^{-1}, \quad W=1.1042\cdot 10^2.
\end{align*}

Let $\hat{a}$ be the zero 
of $F$ obtained from Theorem~\ref{thm:radii} and let $\hat{u}$ be the corresponding solution of~\eqref{eq:SBeq} with Fourier coefficients $\hat{a}$. It then follows from the error bound in Theorem~\ref{thm:radii} that 
\begin{equation*}
    \|\hat{u}-\bar{u}\|_\infty \leq  
	 \|\hat{a}-\bar{a}\|_\nu \leq 
	 3.5069 \cdot 10^{-8}. \qedhere
\end{equation*}
\end{proof}
 
\begin{remark}\label{rem:choiceparam}
The choice of $\nu=(1+10^{-7},1+10^{-7})$ in the proof of Theorem~\ref{thm:examplesol} was made to ensure that solutions are smooth, although a straightforward bootstrap argument also allows the choice $\nu=(1,1)$. For $\bar{\nu}$ and $\bar{\rho}$, we have the relationship $\bar{\nu} = (e^{\bar{\rho}_1},e^{\bar{\rho}_2})$. The choice for $\bar{\nu}=(1.1,1.1)$ was made using experimentation, taking into account that it needs to be chosen in between $(1,1)$ and $\nu_{\bar{a}}$, where $\nu_{\bar{a}}$ is the approximate decay rate of $\bar{a}$, i.e.\ $|\bar{a}_n| \propto \nu_{\bar{a}}^{-n}$ for moderately large $n$.
Concerning the values of the different $N$ that are presented in Table \ref{tab:examplesol}, we have chosen $N^\mathrm{Gal}$ (and the other values of $N$) relatively large in order to obtain a tight error bound. 
Somewhat smaller $N$ still lead to a proof, but with an error bound that is less tight, see also the discussion below.
In general, one may want to choose the different $N$ as small as possible to keep the computation time manageable, but sufficiently large such that the proof is successful and all the requirements (e.g.~$N^\mathrm{Col}\geq N^\mathrm{Jac}$) are met. An overview of the roles of the different $N$ and the corresponding requirements can be found in Appendix \ref{app:overviewN}.
\end{remark}

\begin{remark}\label{rem:rstar}
Equation~\eqref{eq:What} provides a bound $W =\widehat{W}e^{r^\ast}$, which
implies that for a successful proof we essentially need that
$\hat{p}(r):=\frac{1}{2}\widehat{W}e^{r} r^2-(1-Z)r+Y<0$. Note that this
condition is related to the conditions in Theorem~\ref{thm:radii}, as the lower
bound on $r$ in Theorem~\ref{thm:radii} is one of the roots of 
$p(r):=\frac{1}{2}W r^2-(1-Z)r+Y$. 
To obtain a negative value for $p$ and to bring us
strictly in the formulation of Theorem~\ref{thm:radii}, we choose a value
$r^\ast$ such that $\hat{p}(r^\ast)$ is minimal (numerically). Subsequently, we check whether
for this value of $r^\ast$ and some $r<r^\ast$ it indeed holds that
$p(r)=\frac{1}{2}\widehat{W}e^{r^\ast} r^2-(1-Z)r+Y<0$. For the situation in
Theorem~\ref{thm:examplesol}, this procedure resulted in $r^\ast=6.3605\cdot
10^{-3}$.
\end{remark}

For the discussion in this section, the error bound obtained from Theorem~\ref{thm:radii} on the difference between the approximate and the true solution is denoted by $r^\mathrm{min}$. For example, in the proof of Theorem \ref{thm:examplesol} we have $r^\mathrm{min}=3.5069\cdot 10^{-8}$. Besides Theorem \ref{thm:examplesol}, we present five more existence proofs to illustrate the efficacy of the method. The results are combined in Theorem~\ref{thm:othersol}. We discuss the different solutions found after the statement of the theorem. We note that the approximate solutions in the plots in Figure \ref{fig:overviewplots} are extended to domains such that a single period is visualized. 

\begin{theorem}\label{thm:othersol}
For the different combinations of $c$, $q$ and $r^\mathrm{min}$ as presented in Table \ref{tab:overviewproofs}, there exists a solution $\hat{u}$ to \eqref{eq:SBeq} with boundary conditions \eqref{eq:SBBC} such that
$\|\hat{u} - \bar{u}\|_\infty \leq r^\mathrm{min}$.
Here $\bar{u}$ are the numerical approximations visualized in Figure \ref{fig:overviewplots}. The finitely many non-vanishing Fourier-cosine coefficients of $\bar{u}$ can be found in the code at~\cite{github}. 
\end{theorem}
\begin{proof}
The proofs of these solutions are entirely analogous to the proof of Theorem~\ref{thm:examplesol} presented above.
The values of the parameters not presented in Table \ref{tab:overviewproofs}, such as the values of $\nu$, $\bar{\rho}$ and the values of the other $N$, can be found in the code at~\cite{github}.
\end{proof}

\begin{figure}
\vspace{-2\baselineskip}
\begin{subfigure}[t]{0.45\textwidth}
\includegraphics[width=\linewidth]{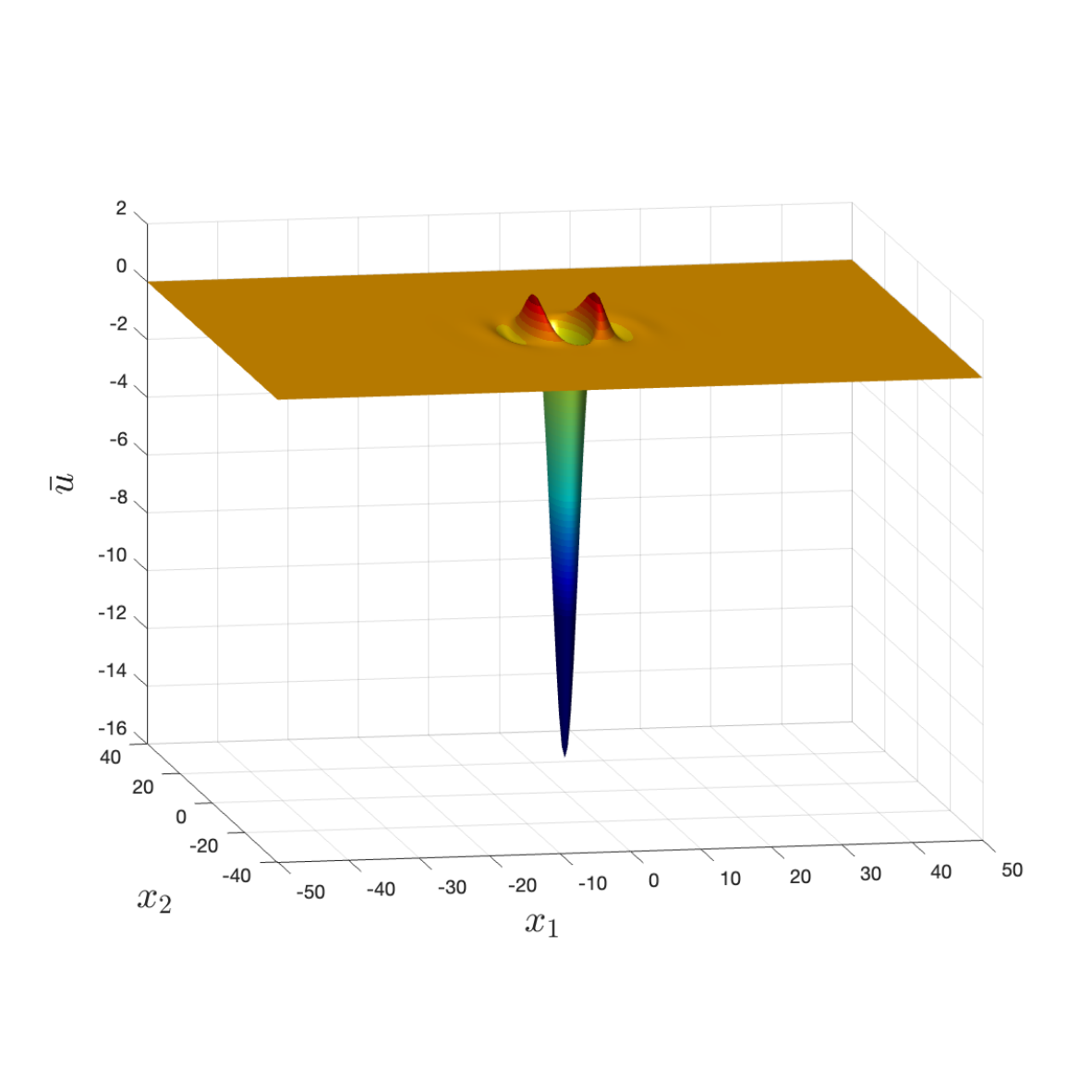}
\vspace{-2\baselineskip}
\caption{Approximate solution with fewer modes for $c=1.1$.} \label{fig:a}
\vspace{-2\baselineskip}
\end{subfigure}\hspace*{\fill}
\begin{subfigure}[t]{0.45\textwidth}
\includegraphics[width=\linewidth]{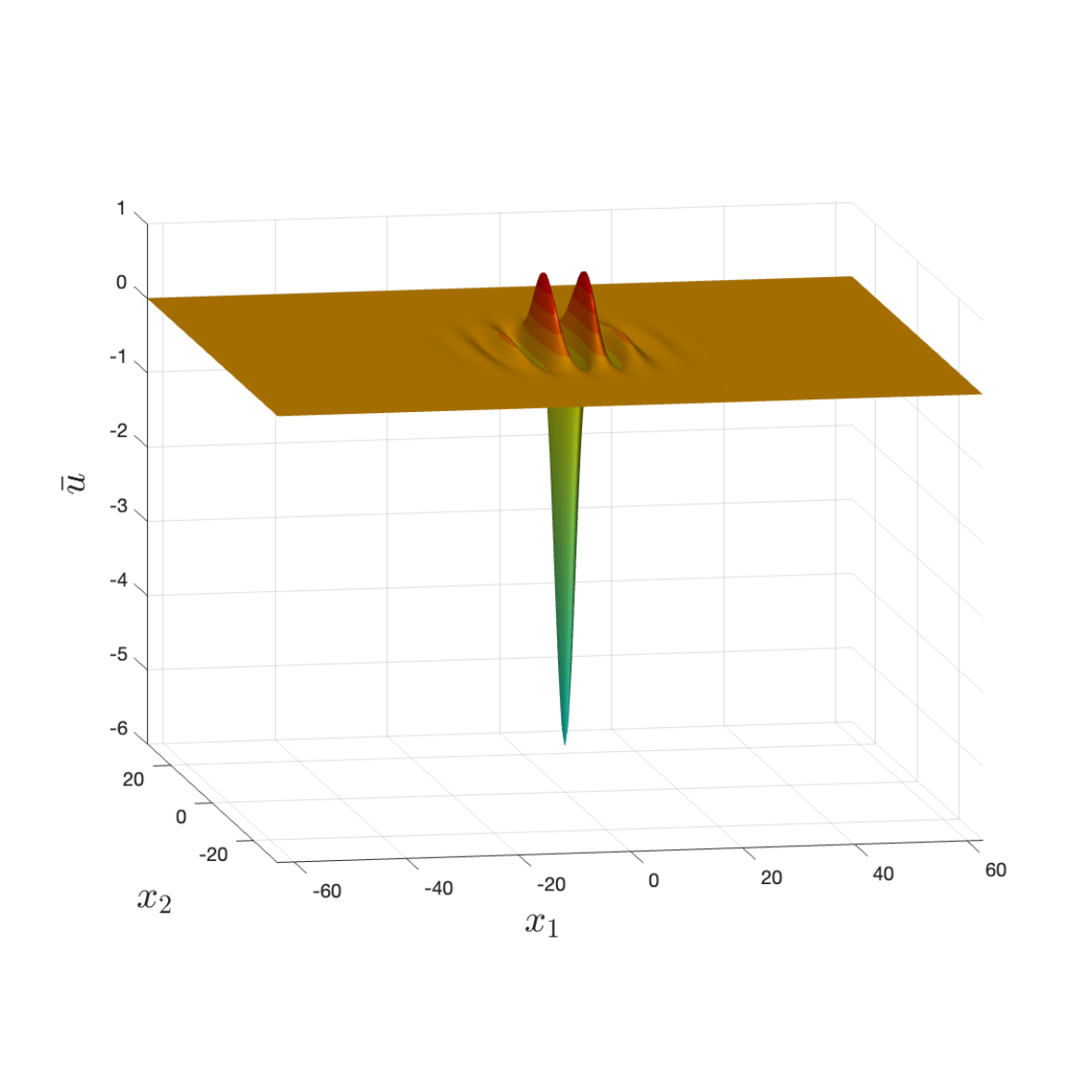}
\vspace{-2\baselineskip}
\caption{Approximate one peak solution for $c=1.3$.} 
\vspace{-2\baselineskip}\label{fig:b}
\end{subfigure}
\begin{subfigure}[t]{0.45\textwidth}
\includegraphics[width=\linewidth]{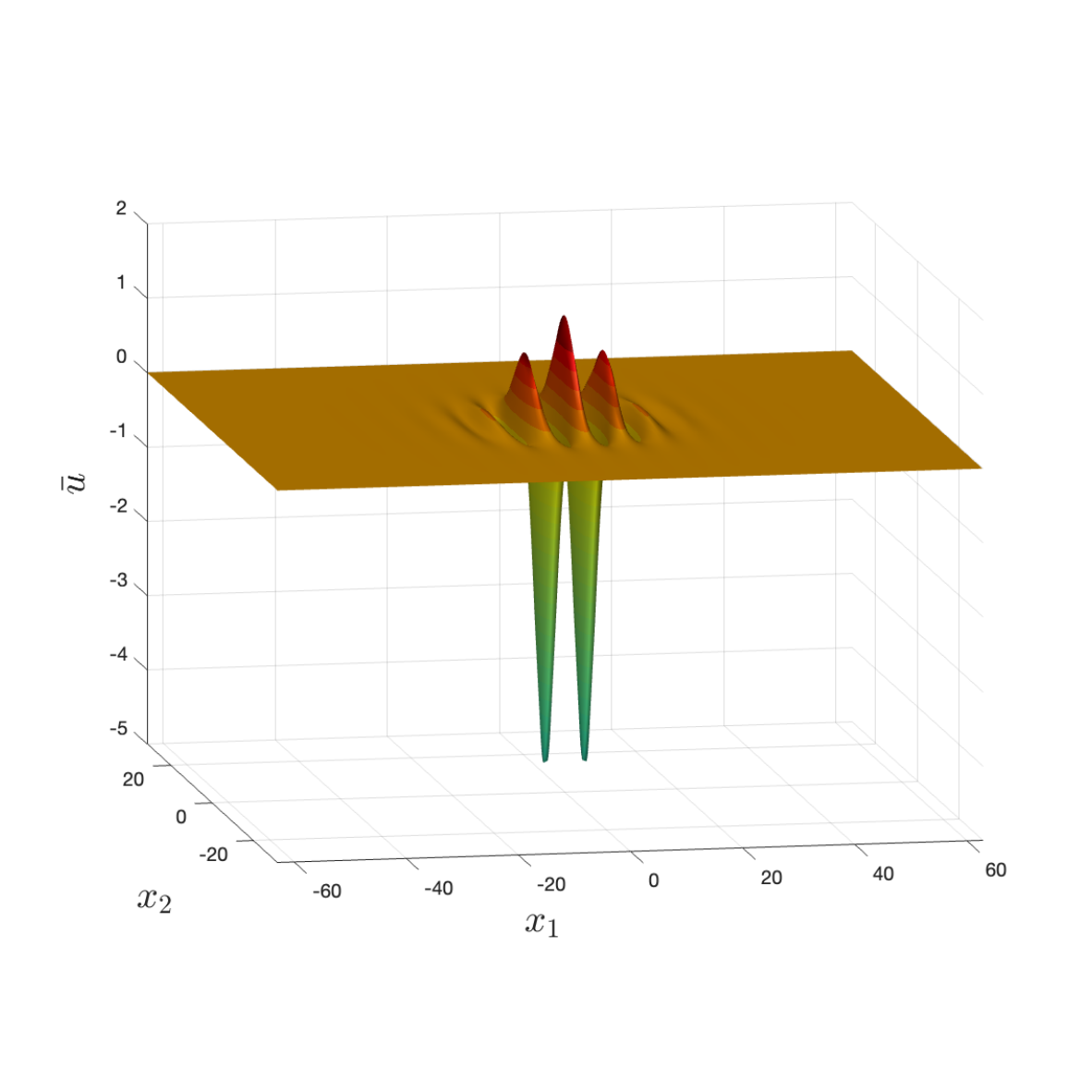}
\caption{Approximate two peak solution for $c=1.3$.} 
\vspace{-2\baselineskip}
\label{fig:c}
\end{subfigure}\hspace*{\fill}
\begin{subfigure}[t]{0.45\textwidth}
\includegraphics[width=\linewidth]{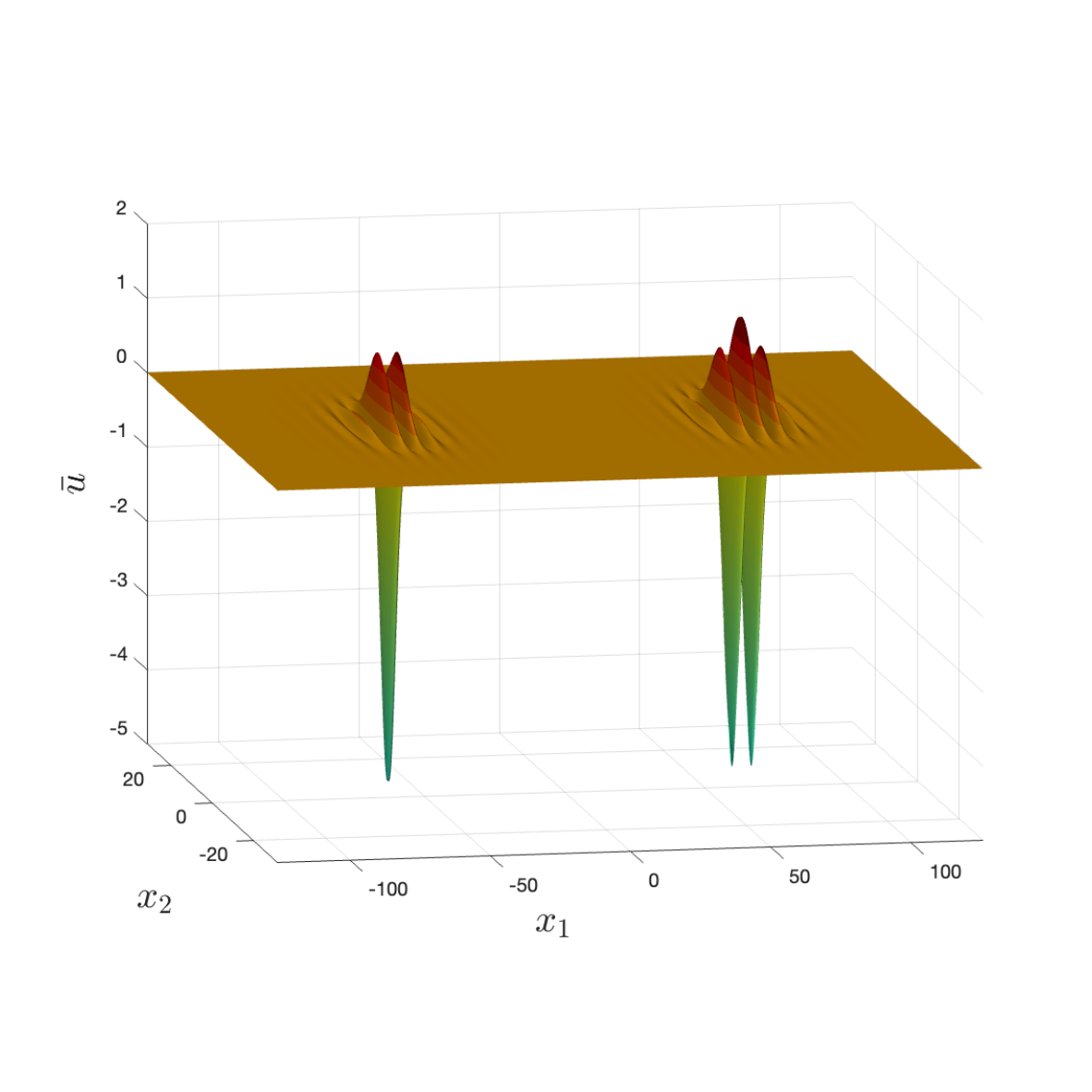}
\caption{Approximate combination solution for $c=1.3$.} \vspace{-2\baselineskip}
\label{fig:d}
\end{subfigure}
\begin{subfigure}[t]{0.45\textwidth}
\includegraphics[width=\linewidth]{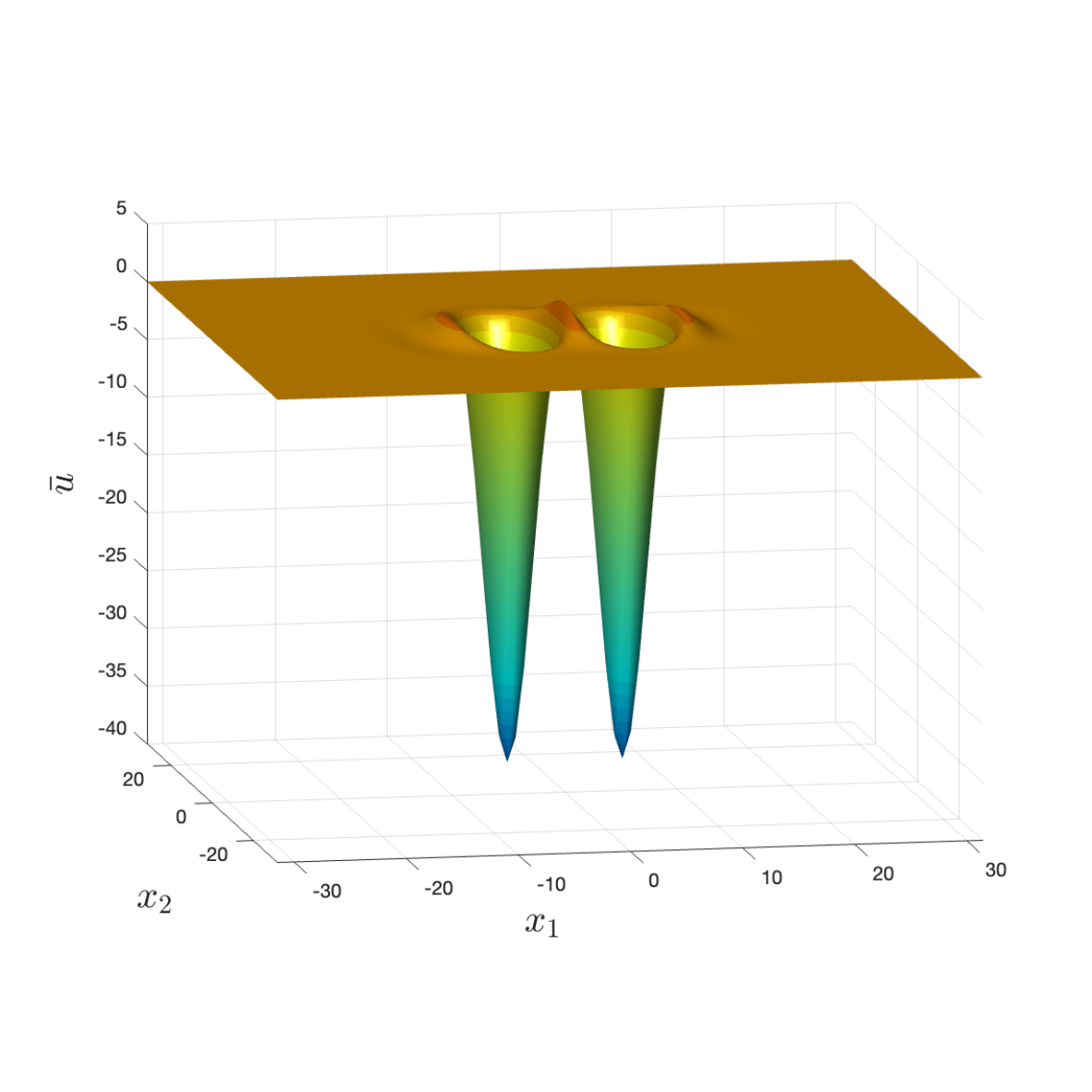}
\vspace{-2\baselineskip}
\caption{Approximate solution for low wave speed $c=0.9$.}
\label{fig:e}
\end{subfigure}\hspace*{\fill}
\begin{subfigure}[t]{0.45\textwidth}
\includegraphics[width=\linewidth]{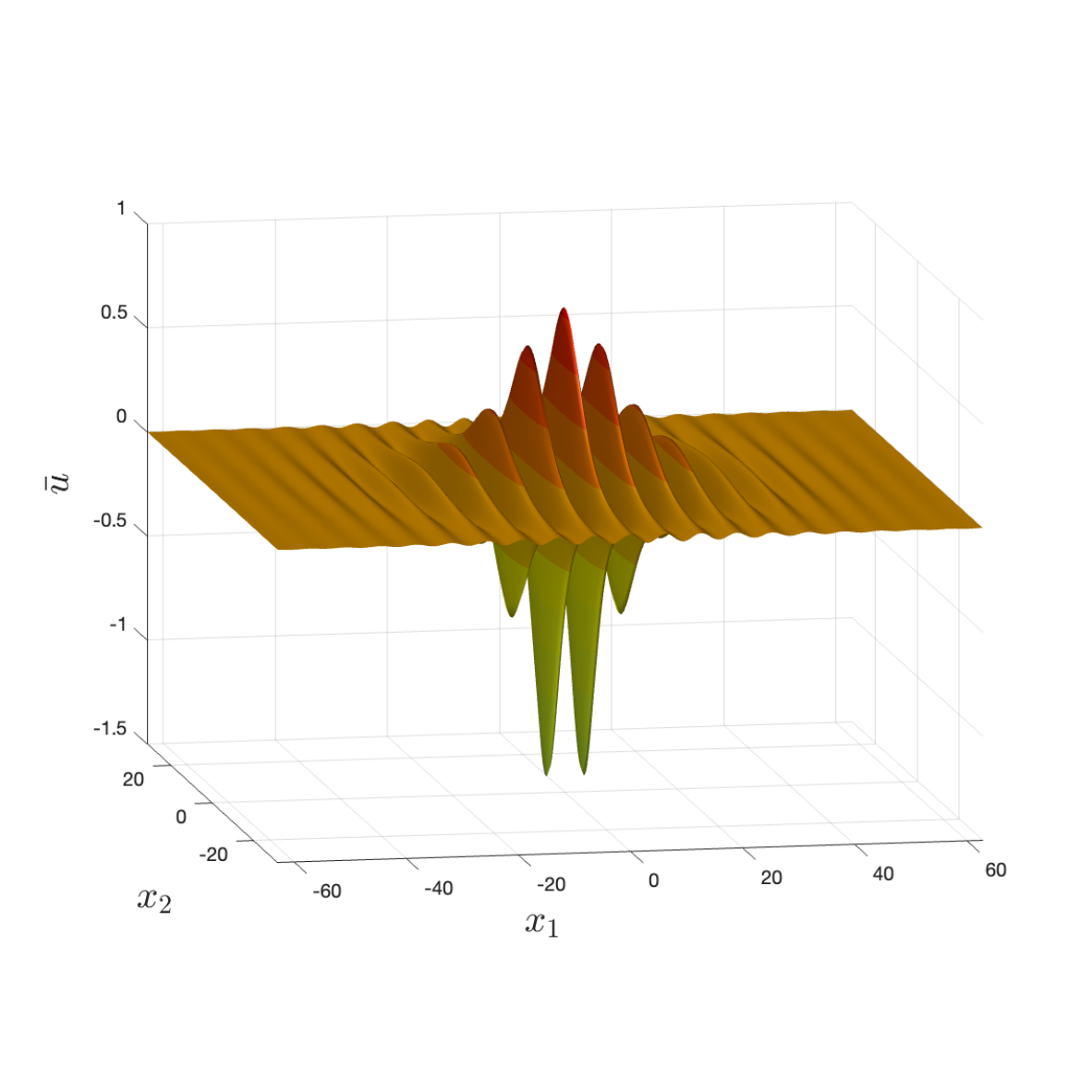}
\vspace{-2\baselineskip}
\caption{Approximate solution for high wave speed $c=1.4$.} \label{fig:f}
\end{subfigure}
\caption{Overview of approximate solutions to \eqref{eq:SBeq} with boundary conditions \eqref{eq:SBBC}, which can be verified using the code in \cite{github}.} 
\vspace{-2\baselineskip}\label{fig:overviewplots}
\end{figure}

\begin{table}[]
\centering
\renewcommand{\arraystretch}{1.2}
\begin{tabular}{|l|l|l|l|l|r|}
\hline
\multicolumn{1}{|c|}{\textbf{Figure} } &
\multicolumn{1}{|c|}{\bm{$N^\mathrm{Gal}$}} &
\multicolumn{1}{|c|}{\bm{$c$}} &
\multicolumn{1}{|c|}{\bm{$q$}} &
\multicolumn{1}{|c|}{\bm{$r^{\min}$}} &
\multicolumn{1}{|c|}{\textbf{time (s)}}
\\ \hline
   \ref{fig:a}      &   $(65,65)$                        &    $1.1$              &   $(0.0628,0.0785)$  &          $3.5069 \cdot 10^{-8}$   &     $7979$    \\ \hline
          \ref{fig:b}    &     $(60,20)$          &     $1.3$         &   $(0.05,0.1)$               &        $2.0794\cdot 10^{-3}$ & $352$  \\ \hline
        \ref{fig:c}                       &     (100,60)   & $1.3$     &    $(0.05,0.1)$              &    $2.1378\cdot 10^{-6}$ &    $4906$                   \\ \hline
        \ref{fig:d}          &    $(200,49)$     &    $1.3$          &       $(0.025,0.1)$           &  $4.2604\cdot10^{-6}$       &    $44330 $             \\ \hline
        \ref{fig:e}         &   $(60,60) $          &        $0.9$      &    $(0.1,0.1)$              &  $2.7351\cdot 10^{-5}$    &        $42935$              \\ \hline
          \ref{fig:f}        &         $(80,40)$     &        $1.4$      &   $(0.05,0.1)$               &   $4.2865\cdot 10^{-6}$&   $4292$                     \\ \hline
\end{tabular}
\caption{Overview of the parameters $N^\mathrm{Gal}$, $c$ and $q$, the difference $r^{\min}$ between the approximate and true solution, and the computation time for the different (approximate) solutions plotted in Figure \ref{fig:overviewplots}.}
\label{tab:overviewproofs}
\end{table}

Figure~\ref{fig:a} essentially reproduces the profile from Figure~\ref{fig:examplesol}. In Figures~\ref{fig:b}-\ref{fig:d} we depict three (approximate) solutions for wave speed $c=1.3$, which is the same value as was used in~\cite{wave_breuer} for results in one dimension. The frequencies $q$ are selected such that the solutions are convincingly localized relative to the domain size. 
In Figures~\ref{fig:b} and~\ref{fig:c} we showcase periodically localized solutions with a single peak and a double peak, respectively. We prove that we can also find a traveling wave solution that combines the one- and two-peak patterns, as shown in Figure~\ref{fig:d}. 
Note that the first component of the frequency corresponding to Figure~\ref{fig:d} is half of the ones in Figures~\ref{fig:b} and~\ref{fig:c}. 
One may conjecture that solutions with more peaks also exist, as well as many combinations of multi-peak patterns, but we leave this for future investigation. 

While an in-depth analysis of automated or near-optimal choices of $N$ is beyond the scope of the present paper, we highlight some features of the presented sample results. For the solution in Figure~\ref{fig:b}, 
the number of Fourier modes are chosen relatively low to minimize computation time. For that purpose, it is especially important to minimize $N^\mathrm{Col}$ and, to a lesser but still significant extent, $N^\mathrm{Jac}$, as they heavily influence the computation time. The flip side of the coin is that the error bound $r^\mathrm{min}$ for this solution is comparatively large, although it is straightforward and computationally not too costly to decrease $r^\mathrm{min}$ by increasing $N^\mathrm{Gal}$ (and $N^\mathrm{Alias}$).
For the rest of the figures we have chosen more Fourier modes than necessary in order to obtain tighter bounds at higher computational cost.

Next, in Figures~\ref{fig:e} and~\ref{fig:f} we illustrate the influence of the wave speed on the size of the periodically localized pattern. For relatively low wave speed $c=0.9$ the amplitude is large (Figure~\ref{fig:e}), whereas for relatively large wave speed $c=1.4$ the amplitude is much smaller and the wave is markedly less localized, see Figure~\ref{fig:f}. For Figure~\ref{fig:e} we choose relatively large $q=(0.1,0.1)$ as at this low wave speed the pattern is more localized than at higher wave speeds.
In general, proving solutions with a lower wave speed requires more modes, as they are needed to approximate the rather spiky solution accurately.

The approximate solutions leading to the rigorously verified solutions corresponding to Figures~\ref{fig:a}-\ref{fig:f} were obtained using Newton's method. Finding a suitable initial guess for Newton's method is often a nontrivial task. A systematic approach for obtaining such guesses is provided by the mountain pass algorithm (see \cite{SB_horak}). However, implementing this algorithm can be quite challenging in practice. Therefore, we adopted an alternative strategy. As a starting point, we constructed initial guesses by extending one-dimensional solutions to the ODE~\eqref{eq:1DSB}, which are easier to find, into two dimensions. Specifically, these one-dimensional solutions were embedded into a two-dimensional domain and multiplied by a smooth bump function with compact support in the $X_2$-direction. To obtain convergence to nontrivial two-dimensional solutions, we added random perturbations to these functions before applying Newton's method. After some trial and error we obtained a nontrivial solution of the type in Figures~\ref{fig:a} and~\ref{fig:b}: a pattern with one large negative peak. Once we found one solution, we used continuation in $c$ to compute solutions of the same type for different wave speeds. By continuing in $c$ up to approximately $c=\sqrt{2}$, we obtained a small amplitude solution $u_{\text{small}}$. In this parameter range, $-u_{\text{small}}$ is almost a solution as well, but this one has a different structure: it features two large negative peaks. Starting with $-u_{\text{small}}$ and decreasing $c$, we arrived at the solutions shown in Figures~\ref{fig:c}, \ref{fig:e} and~\ref{fig:f}. Finally, to construct a third type of solution, we combined these two types of patterns by placing them side by side, sufficiently far apart, and then we applied Newton iterations. This procedure led to the pattern shown in Figure~\ref{fig:d}.

\section{An alternative method: power series}\label{sec:powerseries}

Besides the procedure described in the main part of this paper, we explored another way of proving solutions to the suspension bridge equation \eqref{eq:SBeq} with boundary conditions \eqref{eq:SBBC}. This alternative method is based on the power series expansion of the exponential function
\begin{align}\label{eq:defexp}
    e^u-u-1=\sum_{n=2}^\infty \frac{u^n}{n!}.
\end{align}
By substituting \eqref{eq:defexp} in \eqref{eq:SBeq},
we end up with the zero-finding problem
\begin{align*}
    F_n(a):=\lambda_na_{n}+ \sum_{k=2}^\infty \frac{(\overbrace{a\ast \dots \ast a}^k)_n}{k!}=0 \qquad\text{for all } n\in \mathbb{N}^2,
\end{align*}
where $\lambda_n$ is as in \eqref{eq:deflambda}. Again, we choose $X=\ell^1_\nu$ for our Banach space. 
Fixing some $M \in \mathbb{N}$, we split the series in $F$ in two parts, namely $F:=F^{(M)}+F^{(\infty)}$, where, for $n\in \mathbb{N}^2$,
\begin{equation*}
    F_n^{(M)}({a}) := 
    \lambda_n{a}_n+\sum_{k=2}^M\frac{(\overbrace{{a}\ast \dots \ast{a}}^k)_n}{k!}, \qquad    F^{(\infty)}_n({a}) := \sum_{k=M+1}^\infty\frac{(\overbrace{{a}\ast \dots \ast{a}}^k)_n}{k!}.
\end{equation*}
We can compute all nonzero components of $F_n^{(M)}(\bar{a})$ for fixed $M$ since $F_n^{(M)}(\bar{a})$ vanishes for 
$n \notin I^+_{M N^{\mathrm{Gal}}}$. 
The norm of the tail term $F_n^{(\infty)}(\bar{a})$ we can estimate by
\begin{align}\label{eq:Finf}     \|F^{(\infty)}(\bar{a})\|_{\nu}&=\left\Vert\sum_{k=M+1}^\infty\frac{(\overbrace{\bar{a}\ast \dots \ast \bar{a}}^k)}{k!}\right\Vert_{\nu}\leq \sum_{k=M+1}^\infty\frac{\|\bar{a}\|_{\nu}^k}{k!}=e^{\|\bar{a}\|_{\nu}}-\sum_{k=0}^M\frac{\|\bar{a}\|_{\nu}^k}{k!},
\end{align} 
where the inequality follows from the triangle inequality and the Banach algebra property \eqref{eq:banachalgebra}. Since $\bar{a}$ only has finitely many nonzero terms, we can compute $\|\bar{a}\|_\nu$, hence the right-hand side of \eqref{eq:Finf} can be calculated.

Using this splitting of $F$, all bounds in Theorem \ref{thm:radii} can be computed explicitly. For example, considering the $Y$-bound, instead of computing $\|AF(\bar{a})\|_{\nu}$, we use the triangle inequality to obtain
\begin{align}\label{eq:AFsplit}
    \|AF(\bar{a})\|_{\nu}&\leq \|AF^{(M)}(\bar{a})\|_{\nu}+\|AF^{(\infty)}(\bar{a})\|_{\nu}\nonumber \\
    &\leq\|AF^{(M)}(\bar{a})\|_{\nu}+\|A\|_{B\left(X\right)}\left(e^{\|\bar{a}\|_{\nu}}-\sum_{k=0}^M\frac{\|\bar{a}\|_{\nu}^k}{k!}\right).
\end{align}
We can compute $\|AF^{(M)}(\bar{a})\|_{\nu}$ since $AF^{(M)}(\bar{a})$ has only finitely many nonzero terms (essentially following the same argument as for $F^{(M)}(\bar{a})$). The second term in \eqref{eq:AFsplit} involves $\|A\|_{B\left(X\right)}$, which we can compute using \eqref{eq:Acomp}, and the right-hand side of \eqref{eq:Finf}, which we also can compute, as explained earlier.

\begin{figure}[tb]
    \centering
    \includegraphics[width=0.5\textwidth]{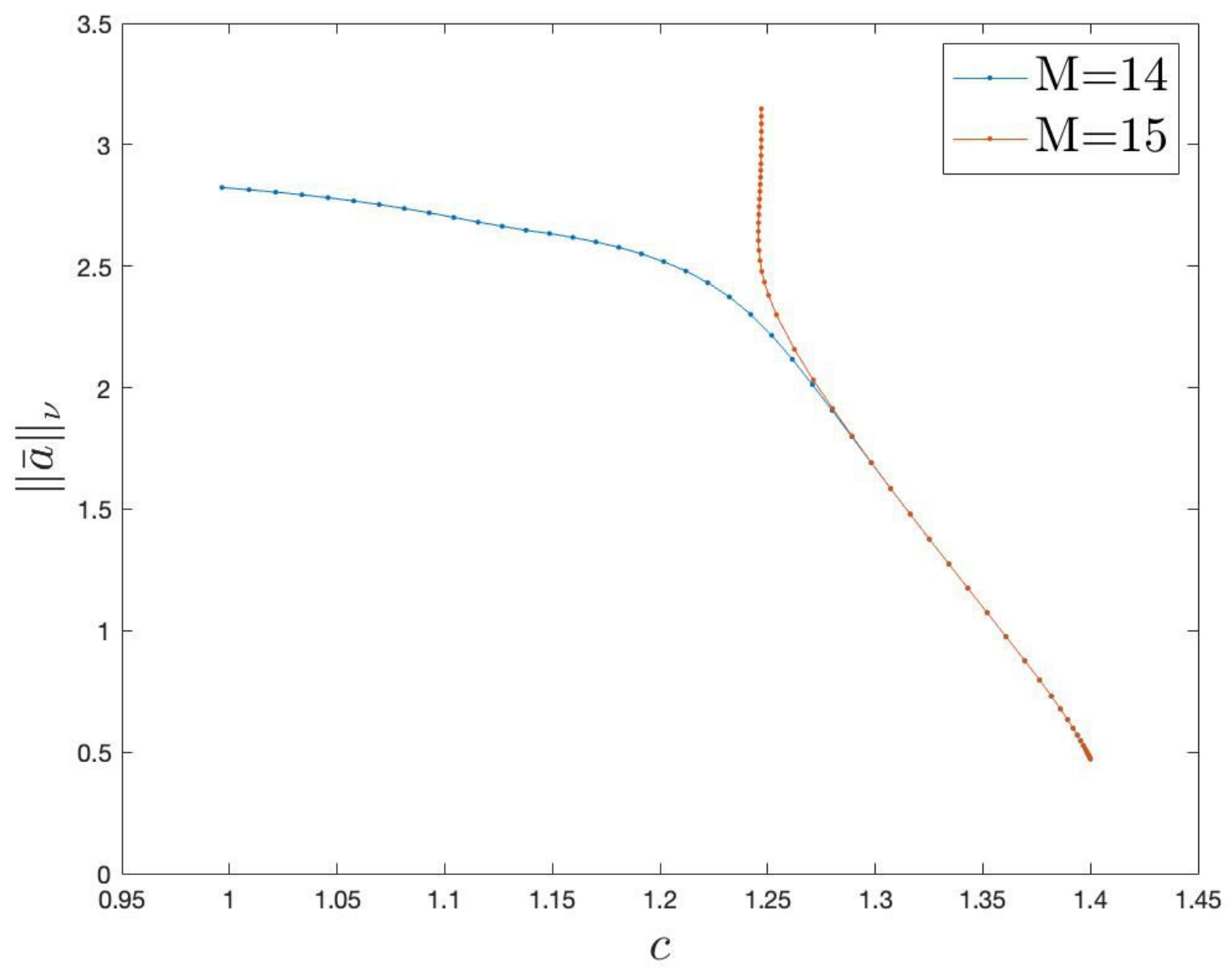}
    \caption{Influence of the parity of $M$. The code used to create this plot can be found in \cite{github}.}
    \label{fig:arccont}
\end{figure}

The computation of the $Z$-bound is similar to the $Y$-bound. First, $DF$ is split into $DF^{(M)}$ and $DF^{(\infty)}$. The operator norm of the polynomial term $DF^{(M)}(\bar{a})$ is then evaluated by computer, whereas the norm of $DF^{(\infty)}(\bar{a})$ is estimated using the Banach algebra property. For the $W$-bound, we use the naive bound described in Section \ref{sec:Wbound}, i.e., $W=\|A\|_{B(X)}e^{\|\bar{a}\|_\nu+r^\ast}$.

This method is more straightforward to understand and implement in comparison with the technique using the discrete Fourier transform and the aliasing error bound in the previous sections. The method described in this section works well for wave speed values close to $\sqrt{2}$, where the amplitude of the solution is small and the Taylor expansion thus converges quickly. Approximating the solution becomes increasingly difficult for smaller values of the wave speed~$c$, as the amplitude of the wave increases. In order to illustrate this, we created Figure \ref{fig:arccont} using numerical continuation. We depict (the norm of) the numerical approximation of a solution, where we performed the numerics with $F$ replaced by $F^{(M)}$ for both $M=14$ and $M=15$. Since the results are completely divergent for $c$ smaller than roughly $1.25$, for such values of~$c$ we know that the numerical zeros of $F^{(M)}$ are in fact not good approximations of a zero of~$F$. Increasing~$M$ further did not improve the situation. Clearly then, approximating $F$ by a Taylor polynomial~$F^{(M)}$, which is essentially what we try to do in this section, already fails at the numerics level, hence attempting a proof based on the Taylor polynomial approximation will be futile. 

We conclude that for $c \gtrsim 1.3$ the alternative approach in this section works nicely and it may even be preferable when studying solutions bifurcating from the uniform state. However, for smaller $c$, and corresponding larger wave profiles, the method based on the FFT is superior. Moreover, that method extends directly to analytic nonlinearities other than the exponential, and hence is more versatile in that respect as well.

\newpage

\appendix
\section{Overview of computational parameters $N$}\label{app:overviewN}
An overview of the different computational parameters $N$ used throughout this paper is presented here.
\begin{itemize}
\item $N^\mathrm{Gal}$:
 As we can only compute finitely many Fourier coefficients, we choose a truncation dimension $N^\mathrm{Gal}$. This   determines how many (nonzero) Fourier coefficients we consider in our numerical approximation. Hence, we have that $\bar{a}_n=0$ for $n \notin I^+_{N^\mathrm{Gal}}$.

\item $N^\mathrm{Jac}$:
Using the computer, we calculate a finite part 
$D \Pi^{N^\mathrm{Jac}}{F}\left(\Pi^{N^\mathrm{Jac}}\bar{a}\right)$
 of the full Jacobian. 
Increasing  $N^\mathrm{Jac}$ can lead to a lower value for the $Z$-bound, which in turn increases the chances of obtaining a successful proof. Typically $N^\mathrm{Jac}=N^\mathrm{Gal}$ is chosen in implementations (see Remark \ref{rem:Njac}), but this is not required. To solve our two-dimensional suspension bridge equation, we require that $N^\mathrm{Jac}$ satisfies the two inequalities in~\eqref{e:Njacmin}.

\item $N^\mathrm{Alias}$:
We construct two methods for enclosing $\bar{b}_n = G_n(\bar{a})$. One of them can be used for general $n\in \mathbb{Z}^2$ (see Lemma \ref{lemma:boundb}) and the other one for $n \in I_{N^\mathrm{Alias}}$, which provides a tighter enclosure (at additional computational cost),  see Lemma \ref{lemma:alias}. Hence, $N^\mathrm{Alias}$ determines from which point on we use the general bound. For implementation purposes, we set $N^{\mathrm{Alias}} \geq N^{\mathrm{Jac}}$. We use the same choice of $N^{\mathrm{Alias}}$ for $\bar{b}$, $\bar{b}'$ and $\bar{b}''$.

\item $N^\mathrm{FFT}$: 
When applying the discrete (inverse) Fourier transformations we use $2N^\mathrm{FFT}$ discretization points or modes. In particular, this determines the size (and values) of $\bar{b}_n^\mathrm{FFT}$ in \eqref{eq:aliasingerror}. We require that both $N_1^\mathrm{FFT}$ and $N_2^\mathrm{FFT}$ are powers of $2$ for algorithmic reasons.
 Furthermore, we require 
 $N^\mathrm{FFT} \gg N^\mathrm{Alias}$ to get tight enclosures on $\bar{b}_n$ for $n\in I^+_{N^{\mathrm{Alias}}}$.

\item $N^\mathrm{Col}$:
We want to compute the norm $\|I-ADF(\bar{a})\|_{B(X)}$ in Section~\ref{sec:Zbound}. For that purpose, we need to estimate $\|[I-ADF(\bar{a})]e_k\|_\nu$ for every $k$. We have two method for this, one for ``small'' $k$ and one for ``large'' $k$. The switching point is at $N^\mathrm{Col}$. We require that $N^\mathrm{Col}\geq N^\mathrm{Jac}$ for the analysis in Section \ref{sec:Zboundsmalln}. 

\item $N^\mathrm{Row}$: 
To evaluate $\|[I-ADF(\bar{a})]e_k\|_\nu$, we analyze $\left([I-ADF(\bar{a})]e_k\right)_n$. As explained in the previous bullet point in this list, we have two ways of analyzing this depending on whether $k\in I_{N^\mathrm{Col}}$ or $k\notin I_{N^\mathrm{Col}}$. Likewise, we have two ways of analyzing $\left([I-ADF(\bar{a})]e_k\right)_n$ depending on whether $n$ is ``small'' or ``large'' and the switching point is at $N^\mathrm{Row}$. 
We require that $N^\mathrm{Row}\geq N^\mathrm{Jac}$ for the analysis in Section \ref{sec:Zboundlargen}.

\item $N^\mathrm{Tail}$: 
When we estimate $\left([I-ADF(\bar{a})]e_k\right)_n$ uniformly for large $k$,
we identify an intermediate range of values of $n$ where we can use computational power to improve on the general estimate that holds for any large $n$. This intermediate range is denoted by $N^\mathrm{Tail}$ and we require that $N^\mathrm{Jac} \leq N^\mathrm{Tail} \leq N^\mathrm{Col}$.
\end{itemize}

\section{Estimation of the constant in Lemma \ref{lemma:boundb}}\label{app:calculationC}
In this appendix, a lemma is proven stating an efficient procedure, based on
the interval-valued discrete Fourier transform, for rigorously estimating the
constant $C_{\bar{\rho}}$ defined in \eqref{eq:defCrho}.
Recalling \eqref{eq:defC}, we can use these constants to estimate the constant
$\widehat{C}_{\bar{\rho}}$ in Lemma \ref{lemma:boundb} for bounding
$\bar{b}_n$. Before stating the lemma and its proof, we introduce some notation.

As in Section~\ref{sec:alias} we choose a discrete Fourier transform with grid dimensions $N^{\mathrm{FFT}} \in \mathbb{N}^2$.
It is not necessary that the choice is the same as in Section~\ref{sec:alias}, although it is in our code. Assuming that $\bar{a}_n$ vanishes for $n \notin I_{N^{\mathrm{Gal}}}$, in this section we require that $N^{\mathrm{FFT}} > N^{\mathrm{Gal}}$, which is a different condition than the one in Section~\ref{sec:alias}.
We recall the notation
\[
J_{N^\mathrm{FFT}} = \{-N_1^{\mathrm{FFT}}\leq n_1 \leq N_1^{\mathrm{FFT}}-1,-N_2^{\mathrm{FFT}}\leq n_2 \leq N_2^{\mathrm{FFT}}-1\},
\] 
and the uniform mesh (of the square $[-\pi,\pi] \times [-\pi,\pi]$).
\begin{align*}
    x_{1,k_1}=\frac{\pi k_1}{N_1^{\mathrm{FFT}}}, \quad   x_{2,k_2}=\frac{\pi k_2}{N_2^{\mathrm{FFT}}} \qquad\qquad \text{ for } k \in J_{N^\mathrm{FFT}}.
\end{align*}
Let $\square=\square_1 \times \square_2$ be the small rectangle with sides 
$\square_j=[0, \frac{\pi}{N_j^{\mathrm{FFT}}}]$, $j=1,2$. 
Then for $\delta \in \square$ we have, since $I_{N^{\mathrm{Gal}}}\subset J_{N^\mathrm{FFT}}$, that
\begin{align}
   \bar{u}(x_k+\delta-i\bar{\rho}) & =  \sum_{n \in J_{N^\mathrm{Gal}}} \bar{a}_{n} e^{i (n_1 (x_{1,k_1}+\delta_1-i\bar{\rho}_1)+n_2 (x_{2,k_2}+\delta_2-i\bar{\rho}_2)} \nonumber \\
   & =  \sum_{n \in J_{N^\mathrm{FFT}}} \bar{a}_{n} e^{n_1  \bar{\rho}_1+n_2 \bar{\rho}_2} e^{i (n_1  \delta_1 + n_2  \delta_2)}
   e^{i (n_1  x_{1,k_1} + n_2  x_{2,k_2})}. \label{eq:interpretfft1}
\end{align}
The right-hand side in~\eqref{eq:interpretfft1} can be interpreted as the discrete Fourier transform of 
$\bar{a}^{\bar{\rho},\delta}$, where 
\[
   \bar{a}^{\bar{\rho},\delta}_n := \bar{a}_{n} e^{n_1  \bar{\rho}_1+n_2 \bar{\rho}_2} e^{i (n_1  \delta_1 + n_2  \delta_2)}.
\]  
By introducing the corresponding interval-valued sequence
\[
   \bar{a}^{\bar{\rho},\square}_n := \bar{a}_{n} e^{n_1  \bar{\rho}_1+n_2 \bar{\rho}_2} e^{i (n_1  \square_1 + n_2  \square_2)},
\]  
and its interval-valued discrete Fourier transform (evaluated efficiently using interval arithmetic)
\[
  \bar{u}^{\bar{\rho},\square}_k 
  := \sum_{n \in J_{N^\mathrm{FFT}}} \bar{a}^{\bar{\rho},\square}_n  
   e^{i (n_1  x_{1,k_1} + n_2  x_{2,k_2})}, 
\]
we then have that 
\begin{align}\label{eq:fftbaru}
   \bar{u}(x_k+\delta-i\bar{\rho})  \in  \bar{u}^{\bar{\rho},\square}_k 
   \qquad \text{for all } \delta\in \square \quad\text{and}\quad k \in J_{N^\mathrm{FFT}}.
\end{align}
We are now ready to state the lemma.
\begin{lemma}
    The constant $C_{\bar{\rho}}$ defined in \eqref{eq:defCrho} can be estimated via 
\begin{align*}
        C_{\bar{\rho}}\leq 
\textup{upperbound} \left(\frac{1}{4N_1^{\mathrm{FFT}}N_2^{\mathrm{FFT}}}\sum_{k \in J_{N^\mathrm{FFT}}} |\mathcal{G}(\bar{u}^{\bar{\rho},\square}_k )| \right), 
   \end{align*}
where the upper bound of the interval is taken.
\end{lemma}

\begin{proof}
Recalling \eqref{eq:defCrho} and~\eqref{eq:fftbaru}, we have that 
\begin{align*}
C_{\bar{\rho}}&=\frac{1}{(2 \pi)^2} \int_{-\pi} ^{\pi} \int_{-\pi}^{\pi}|\mathcal{G}(\bar{u}(x_1-i \bar{\rho}_1,x_2-i \bar{\rho}_2))| d x_1dx_2\\
 &=\frac{1}{(2 \pi)^2} \sum_{k \in J_{N^\mathrm{FFT}}}
\int_{x_{2,k_2}}^{x_{2,k_2+1}} 
\int_{x_{1,k_1}}^{x_{1,k_1+1}} 
|\mathcal{G}(\bar{u}(x_1-i \bar{\rho}_1,x_2-i \bar{\rho}_2))| d x_1dx_2\\
&\leq\frac{1}{4N_1^{\mathrm{FFT}}N_2^{\mathrm{FFT}}} \sum_{k \in J_{N^\mathrm{FFT}}}\underset{\delta\in\square}{\sup}\bigl|\mathcal{G}(\bar{u}(x_{1,k_1}+\delta_1-i \bar{\rho}_1,x_{2,k_2}+\delta_2-i \bar{\rho}_2))\bigr| \\
&\leq \text{upperbound} \left(\frac{1}{4N_1^{\mathrm{FFT}}N_2^{\mathrm{FFT}}}\sum_{k \in J_{N^\mathrm{FFT}}} |\mathcal{G}(\bar{u}_k^{\rho,\square})| \right). \qedhere
\end{align*}
\end{proof}

\section{Two-dimensional discrete Poisson summation formula}\label{app:poisson}

In this appendix, the two-dimensional discrete Poisson summation formula is proven for our case, which was used in Section \ref{sec:alias}.
We assume that we are in the setting described there, hence in particular that
$\mathcal{G}\circ \bar{u}$ is equal to its uniformly  convergent Fourier series.
\begin{lemma}
For $\bar{b}_n$, the Fourier coefficients of $\mathcal{G}\circ \bar{u}$, and $\bar{b}^\mathrm{FFT}_n$ as defined via \eqref{eq:defbbar1}, we have the discrete Poisson summation formula 
    \begin{align}\label{eq:poisson}
\bar{b}^\mathrm{FFT}_{n}=\bar{b}_{n}+\sum_{\substack{j\in\mathbb{Z}^2 \\ (j_1,j_2)\neq (0,0)}}\bar{b}_{n+2j N^{\mathrm{FFT}}}, \qquad \text {for } n \in J_{N^\mathrm{FFT}}.
\end{align}
\end{lemma}

\begin{proof}
    The proof mimics the proof for the one-dimensional variant in Section 6.2 of \cite{poisson_briggs}. In our case, we consider the mesh given in \eqref{eq:mesh}. Then, we have that for $k \in J_{N^\mathrm{FFT}}$

    \begin{align*}
        \mathcal{G}\left(\bar{u}\left(x_{1,k_1},x_{2,k_2}\right)\right)=\sum_{j \in \mathbb{Z}^2}\bar{b}_j e^{i(j_1 x_{1,k_1}+j_2 x_{2,k_2})}.
    \end{align*}
   Recalling \eqref{eq:defbbar1}, we obtain 
    \begin{align*}
        \bar{b}^\mathrm{FFT}_{n} &=\frac{1}{4N_1^{\mathrm{FFT}}N_2^{\mathrm{FFT}}} \sum_{k \in J_{N^\mathrm{FFT}}} \mathcal{G}\left(\bar{u}\left(x_{1,k_1},x_{2,k_2}\right)\right) e^{-i n_1 x_{1,k_1}-in_2 x_{2,k_2}}\\
        &=\frac{1}{4N_1^{\mathrm{FFT}}N_2^{\mathrm{FFT}}}\sum_{j \in \mathbb{Z}^2}\bar{b}_j\sum_{k \in J_{N^\mathrm{FFT}}}e^{i  x_{1,k_1}(j_1-n_1)}e^{i  x_{2,k_2}(j_2-n_2)}\\
        &\overset{\eqref{eq:mesh}}{=}\frac{1}{4N_1^{\mathrm{FFT}}N_2^{\mathrm{FFT}}}\sum_{j \in \mathbb{Z}^2}\bar{b}_j\sum_{k \in J_{N^\mathrm{FFT}}}e^{\frac{\pi i k_1}{N^\mathrm{FFT}_1}(j_1-n_1)}e^{\frac{\pi i k_2}{N^\mathrm{FFT}_2}(j_2-n_2)}.
    \end{align*}
By using the geometric summation, we infer that $\sum_{k \in J_{N^\mathrm{FFT}}}e^{\frac{\pi i k_1}{N^\mathrm{FFT}_1}(j_1-n_1)}e^{\frac{\pi i k_2}{N^\mathrm{FFT}_2}(j_2-n_2)}$ vanishes unless $j_1-n_1$ is a multiple of $2N^\mathrm{FFT}_1$ and $j_2-n_2$ is a multiple of $2N^\mathrm{FFT}_2$. Hence, 
    \begin{align*}
        \bar{b}^\mathrm{FFT}_n =\sum_{j \in \mathbb{Z}^2}\bar{b}_{n_1+2j_1N^\mathrm{FFT}_1,n_2+2j_2N^\mathrm{FFT}_2} \, ,
    \end{align*}
    which can be rewritten to \eqref{eq:poisson}.
\end{proof}

\bibliographystyle{abbrv}

\bibliography{MyBib}

\end{document}